\documentclass{amsart}
\usepackage{amssymb,mathtools}
\usepackage[british]{babel}
\usepackage{enumitem}
\usepackage[latin1]{inputenc}

\usepackage{url}
\usepackage{tikz-cd}

\newtheorem{theorem}{Theorem}
\numberwithin{theorem}{section}
\newtheorem{corollary}[theorem]{Corollary}
\newtheorem{lemma}[theorem]{Lemma}
\newtheorem{proposition}[theorem]{Proposition}
\theoremstyle{definition}
\newtheorem{definition}[theorem]{Definition}

\numberwithin{equation}{section}

\newcommand{\supp}{\operatorname{supp}}
\newcommand{\leqf}{\leq^{\operatorname{fin}}}
\newcommand{\lef}{<^{\operatorname{fin}}}
\newcommand{\po}{\mathsf{PO}}
\newcommand{\wpo}{\mathsf{WPO}}
\newcommand{\rng}{\operatorname{rng}}
\newcommand{\rca}{\mathsf{RCA}}
\newcommand{\seq}{\operatorname{Seq}}
\newcommand{\ang}{\mathbin{\angle}}
\newcommand{\nf}{\mathbin{=_{\operatorname{NF}}}}

\newcommand{\tl}{\vartriangleleft}

\title[Reverse mathematics of a uniform Kruskal-Friedman theorem]{Reverse mathematics of\\ a uniform Kruskal-Friedman theorem}
\author{Anton Freund}

\address{Anton Freund, Department of Mathematics, Technical University of Darmstadt, Schloss\-garten\-str.~7, 64289~Darmstadt, Germany}
\email{freund@mathematik.tu-darmstadt.de}

\thanks{Funded by the Deutsche Forschungsgemeinschaft (DFG, German Research Foundation) -- Project number 460597863.}

\begin{document}

\begin{abstract}
The Kruskal-Friedman theorem asserts: in any infinite sequence of finite trees with ordinal labels, some tree can be embedded into a later one, by an embedding that respects a certain gap condition. This strengthening of the original Kruskal theorem has been proved by I.~K\v{r}\'{i}\v{z} (Ann.~Math.~1989), in confirmation of a conjecture due to H.~Friedman, who had established the result for finitely many labels. It provides one of the strongest mathematical examples for the independence phenomenon from G\"odel's theorems. The gap condition is particularly relevant due to its connection with the graph minor theorem of N.~Robertson and P.~Seymour. In the present paper, we consider a uniform version of the Kruskal-Friedman theorem, which extends the result from trees to general recursive data types. Our main theorem shows that this uniform version is equivalent both to $\Pi^1_1$-transfinite recursion and to a mini\-mal bad sequence principle of K\v{r}\'{i}\v{z}, over the base theory~$\mathsf{RCA_0}$ from reverse mathematics. This sheds new light on the role of infinity in finite combinatorics.
\end{abstract}

\keywords{Kruskal-Friedman theorem, gap condition, $\Pi^1_1$-transfinite recursion, well ordering principles, dilators, reverse mathematics, ordinal analysis}
\subjclass[2020]{03B30, 06A07, 05C83, 03F15, 03F35}

\maketitle

\section{Introduction}

We begin with a precise formulation of the Kruskal-Friedman theorem for trees. Let $\mathcal T[\nu]$ be the set of finite rooted trees with vertex labels from a given ordinal~$\nu$ (identified with the set of its predecessors). Elements of~$\mathcal T[\nu]$ will be written in the form $\alpha\star[t_0,\ldots,t_{n-1}]$, where $\alpha<\nu$ is the label of the root and $t_0,\ldots,t_{n-1}$ are the immediate subtrees. We ignore the order of subtrees but not their multiplicities. To determine a partial order on~$\mathcal T[\nu]$ by recursion, we declare that
\begin{equation*}
s=\alpha\star[s_0,\ldots,s_{m-1}]\leq_{\mathcal T[\nu]}\beta\star[t_0,\ldots,t_{n-1}]=t
\end{equation*}
holds precisely if we have $\alpha\leq\beta$ and one of the following clauses applies:
\begin{enumerate}[label=(\roman*)]
\item there is an injective function $f:\{0,\ldots,m-1\}\to\{0,\ldots,n-1\}$ such that we have $s_i\leq_{\mathcal T[\nu]}t_{f(i)}$ for all~$i<m$,
\item we have $s\leq_{\mathcal T[\nu]}t_j$ for some $j<n$.
\end{enumerate}
Note that clause~(i) corresponds to an embedding that sends the root to the root and subtrees of~$s$ into different subtrees of~$t$, which means that infima with respect to the tree order are preserved. In clause~(ii), the root of~$s$ is sent to an internal node~of~$t$. All labels below this node must be at least~$\alpha$, as the above condition~$\alpha\leq\beta$ applies recursively. If~$\alpha_i$ is the root label of $s_i$ in clause~(i), it follows that all labels below the image of $s_i$ in $t_{f(i)}$ are at least~$\alpha_i$ (while we may have $\alpha\leq\beta<\alpha_i$ at the root of~$t$). Recursively, the condition propagates to all `gaps' that the image of~$s$ leaves~in~$t$. Hence our order on~$\mathcal T[\nu]$ coincides with H.~Friedman's notion of strong gap embeddability, as determined by condition (1.2.5) of~\cite{kriz-conjecture} (see also~\cite{simpson85}).

By the Kruskal-Friedman theorem we shall mean the statement that $\mathcal T[\nu]$ is a well partial order ($\wpo$) for any ordinal~$\nu$, i.\,e., that any infinite sequence $t_0,t_1,\ldots$ in $\mathcal T[\nu]$ involves an inequality $t_i\leq_{\mathcal T[\nu]}t_j$ for some~$i<j$. Note that~$\nu=1$ amounts to the case without labels, which was proved by J.~Kruskal~\cite{kruskal60}. The extension~to all finite~$\nu<\omega$ is due to Friedman (see the presentation by S.~Simpson~\cite{simpson85}). It has been used in N.~Robertson and P.~Seymour's proof of their graph minor theorem~(see~\cite{friedman-robertson-seymour}). The Kruskal-Friedman theorem for general~$\nu$ was conjectured by Friedman and proved by I.~K\v{r}\'{i}\v{z}~\cite{kriz-conjecture}. For a somewhat different version of the gap condition, a proof of the general case was given by L.~Gordeev~\cite{gordeev-gap}.

The original Kruskal theorem is unprovable in a fairly strong axiom system~$\mathsf{ATR}_0$ that is associated with `predicative mathematics', while the extension to all $\nu<\omega$ is unprovable in the even stronger system $\Pi^1_1\text{-}\mathsf{CA}_0$, as shown by Friedman~\cite{simpson85}. Together with Robertson and Seymour~\cite{friedman-robertson-seymour}, the latter has deduced that $\Pi^1_1\text{-}\mathsf{CA}_0$ cannot prove the graph minor theorem (not even for bounded tree width). This is a particularly striking manifestation of the incompleteness phenomenon from G\"odel's theorems, for the following two reasons: First, the axiom system $\Pi^1_1\text{-}\mathsf{CA}_0$ is strong enough to prove the vast majority of mathematical theorems that have been analyzed in the research program of reverse mathematics (see~\cite{simpson09} for a comprehensive introduction). Secondly, the graph minor theorem has been described as one of ``the deepest theorems that mathematics has to offer"~\cite[Ch.~12]{diestel-graph-theory}. It is also of great importance for theoretical computer science, as it guarantees that certain tasks can be solved by polynomial time algorithms (see, e.\,g.,~\cite{robertson-seymour-polytime}).

For Kruskal's original theorem, we have two proofs that are optimal in different respects: On the one hand, C.~Nash-Williams~\cite{nash-williams63} has given an extremely elegant proof, for which he introduced his influential notion of `minimal bad sequence'. This proof, however, uses axioms that are stronger than necessary. Specifically, A.~Marcone~\cite{marcone-bad-sequence} has shown that a fundamental existence principle for minimal bad sequences is equivalent to $\Pi^1_1$-comprehension, the main axiom of $\Pi^1_1\text{-}\mathsf{CA}_0$. This axiom is strictly stronger than Kruskal's theorem. On the other hand, M.~Rathjen and A.~Weiermann~\cite{rathjen-weiermann-kruskal} have analyzed a proof via `reifications', which uses just the right axioms but appears less elegant. The tension between different forms of optimality -- elegance and the desire to use the weakest possible axioms -- is resolved in a recent paper by Rathjen, Weiermann and the present author~\cite{frw-kruskal}. There we investigate a uniform version of Kruskal's theorem, which extends the original result from trees to general recursive data types (see below for details). By the main result of~\cite{frw-kruskal}, this uniform Kruskal theorem is equivalent to $\Pi^1_1$-comprehension (over $\mathsf{RCA}_0$ with the chain antichain principle). So in the uniform case, the elegant proof via minimal bad sequences uses just the right axioms. This provides formal support for the intuition that Nash-Williams has given `the canonical' proof of Kruskal's theorem. In the present paper, we show that a uniform Kruskal-Friedman theorem is equivalent to $\Pi^1_1$-transfinite recursion (explained below). We also show that both are equivalent to a minimality principle for bad sequences under the gap condition, which K\v{r}\'{i}\v{z}~\cite{kriz-conjecture} has established in proving the Kruskal-Friedman theorem for trees.

Let us now recall the terminology from~\cite{frw-kruskal} that is needed to state the uniform Kruskal-Friedman theorem. To motivate the somewhat abstract definitions, we will use the case of trees as a running example. First, a map $f:X\to Y$ between partial orders is called a quasi embedding if it is order reflecting, i.\,e., if $f(x)\leq_Yf(x')$ implies $x\leq_X x'$. It is called an embedding if it is also order preserving, i.\,e., if the converse implication holds as well. Let us write~$\po$ for the category with the partial orders as objects and the quasi embeddings as morphisms. We say that a functor $W:\po\to\po$ preserves embeddings if $W(f):W(X)\to W(Y)$ is an embedding whenever the same holds for $f:X\to Y$. By $[X]^{<\omega}$ we denote the set of finite subsets of a given set~$X$. To make the construction functorial, we declare that $f:X\to Y$ maps to $[f]^{<\omega}:[X]^{<\omega}\to[Y]^{<\omega}$ with $[f]^{<\omega}(a):=\{f(x)\,|\,x\in a\}$. Let us also agree to write $\rng(f)=\{f(x)\,|\,x\in X\}$ for the range or image. The forgetful functor from partial orders to their underlying sets will be left implicit. Conversely, we will often consider subsets as suborders. We can now recall a central definition from~\cite{frw-kruskal}, which adapts J.-Y.~Girard's~\cite{girard-pi2} notion of dilator on linear~orders.

\begin{definition}\label{def:po-dilator}
A $\po$-dilator is a functor~$W:\po\to\po$ that preserves \mbox{embeddings} and admits a natural transformation $\supp:W\Rightarrow[\cdot]^{<\omega}$ such that the following `support condition' holds: for any embedding $f:X\to Y$ and any $\sigma\in W(Y)$ we have
\begin{equation*}
\supp_Y(\sigma)\subseteq\rng(f)\quad\Rightarrow\quad\sigma\in\rng(W(f)).
\end{equation*}
We call $W$ a $\wpo$-dilator if $W(X)$ is a well partial order whenever~$X$ is one.
\end{definition}

Concerning the support condition, we point out that the converse implication is automatic by naturality, as $\sigma=W(f)(\sigma_0)$ entails
\begin{equation*}
\supp_Y(\sigma)={\supp_Y}\circ W(f)(\sigma_0)=[f]^{<\omega}\circ\supp_X(\sigma_0)\subseteq\rng(f).
\end{equation*}
If the required transformation exists, then $\supp_Y(\sigma)$ is uniquely determined as the smallest $X\in[Y]^{<\omega}$ such that we have $\sigma\in\rng(W(f))$ for the inclusion~$f:X\hookrightarrow Y$. To give an example that will be relevant for the case of trees, we define $W(X)$ as the set of finite multisets $[x_0,\ldots,x_{n-1}]$ with entries $x_i\in X$. Let us declare that
\begin{equation*}
[x_0,\ldots,x_{m-1}]\leq_{W(X)}[y_0,\ldots,y_{n-1}]
\end{equation*}
holds precisely if there is an injective function $f:\{0,\ldots,m-1\}\to\{0,\ldots,n-1\}$ with $x_i\leq_X y_{f(i)}$ for all~$i<m$. Note that the transformation $X\mapsto W(X)$ preserves well partial orders, due to Higman's lemma. We obtain a $\wpo$-dilator by setting
\begin{align*}
W(f)([x_0,\ldots,x_{n-1}])&:=[f(x_0),\ldots,f(x_{n-1})],\\
\supp_X([x_0,\ldots,x_{n-1}])&:=\{x_0,\ldots,x_{n-1}\}.
\end{align*}
For subsets $a$ and $b$ of a partial order~$X$, we shall abbreviate
\begin{equation*}
a\leqf_X b\quad:\Leftrightarrow\quad\text{for each $x\in a$ there is a $y\in b$ with $x\leq_X y$}.
\end{equation*}
We use analogous notation with $<$ at the place of~$\leq$. In the case of singletons we write $x\leqf_X b$ and $a\leqf y$ rather than~$\{x\}\leqf_X b$ and $a\leqf_X\{y\}$, respectively. The multiset construction that we have just described has the following property.

\begin{definition}\label{def:normal}
A $\po$-dilator $W$ is called normal if we have
\begin{equation*}
\sigma\leq_{W(X)}\tau\quad\Rightarrow\quad\supp_X(\sigma)\leqf_X\supp_X(\tau),
\end{equation*}
for any partial order~$X$ and all elements $\sigma,\tau\in W(X)$.
\end{definition}

In the case of linear orders, the given condition ensures that normal dilators induce normal functions on the ordinals (essentially due to P.~Aczel~\cite{aczel-normal-functors}; see also Girard's~\cite{girard-pi2} notion of flower and the discussion in~\cite{freund-rathjen_derivatives}). There is no apparent connection with our use of normality, which is explained below. The following notion will be at the heart of our uniform Kruskal-Friedman theorem.

\begin{definition}\label{def:nu-FP}
For an ordinal~$\nu$ and a normal $\po$-dilator~$W$, a $\nu$-Kruskal fixed point of~$W$ consists of a partial order~$X$ and a bijection $\kappa:\nu\times W(X)\to X$ with
\begin{equation*}
\kappa(\alpha,\sigma)\leq_X \kappa(\beta,\tau)\quad\Leftrightarrow\quad\alpha\leq\beta\text{ and }\left(\,\sigma\leq_{W(X)}\tau\text{ or }\kappa(\alpha,\sigma)\leqf_X\supp_X(\tau)\,\right).
\end{equation*}
We say that $(X,\kappa)$ is initial if any other $\nu$-Kruskal fixed point~$(Y,\pi)$ of~$W$ admits a unique quasi embedding~$f:X\to Y$ such that the diagram
\begin{equation*}
\begin{tikzcd}
\nu\times W(X)\arrow[r,"\kappa"]\arrow[d,swap,"\nu\times W(f)"] & X\arrow[d,"f"]\\
\nu\times W(Y)\arrow[r,"\pi"] & Y
\end{tikzcd}
\end{equation*}
commutes, where $\nu\times W(f)$ denotes the map $(\alpha,\sigma)\mapsto(\alpha,W(f)(\sigma))$.
\end{definition}

Let us remark that the unique~$f$ in the definition is always an embedding, due to the last claim in Theorem~\ref{thm:initial-criterion} below. When $W$ is the multiset construction from above, a $\nu$-Kruskal fixed point is given by the set $\mathcal T[\nu]$ of finite rooted trees with vertex labels from~$\nu$ and the function
\begin{equation*}
\kappa:\nu\times W(\mathcal T[\nu])\to\mathcal T[\nu]\quad\text{with}\quad\kappa(\alpha,[t_0,\ldots,t_{n-1}]):=\alpha\star[t_0,\ldots,t_{n-1}].
\end{equation*}
Here the given value of~$\kappa$ is the tree with root label~$\alpha$ and immediate subtrees~$t_i$. One should observe that the equivalence in Definition~\ref{def:nu-FP} coincides with the recursive characterization of gap embeddability via clauses~(i) and~(ii) above. The given fixed point is initial, since each finite tree can be generated by finitely many applications of the function~$\kappa$ (starting with the empty multiset to generate leaves). Formally, this can be confirmed via condition~(ii) in Theorem~\ref{thm:initial-criterion}. Let us point out that the example motivates all features of our abstract definitions: Functoriality is used to state the universal property that captures the intuitive idea of a recursively generated structure. Supports provide a general notion of immediate subtree, not least in the equivalence in Definition~\ref{def:nu-FP}. Concerning normality, we first note that no tree $s$ can be embedded into a tree~$t$ of smaller height. This can be verified by induction over the recursive definition of embeddability, as given at the beginning of the present paper. In clause~(ii) of that definition, the induction hypothesis will ensure that each subtree~$s_i$ of~$s$ has at most the height of some subtree~$t_{f(i)}$ of~$t$. The role of normality is to provide the function~$f$ on indices. Lemma~\ref{lem:heights} below is a generalization of the given argument. 

In the next section, we will present a construction that yields an initial \mbox{$\nu$-}fixed point of any normal $\po$-dilator. This fixed point must be isomorphic to any other initial $\nu$-fixed point of the same $\po$-dilator, by a standard argument about universal properties (see the paragraph after Corollary~\ref{cor:Kruskal-fp-exist} below). Let us now give a precise definition of an expression that was used above: by a recursive data type with labels from~$\nu$, we shall mean the initial $\nu$-fixed point of some normal $\wpo$-dilator~$W$ (note the requirement that~$W$ preserves well partial orders). As demonstrated by our running example, the \mbox{$\nu$-}labelled trees with gap condition arise as the special case where $W$ is the multiset construction. The following asserts that the Kruskal-Friedman theorem holds not only for trees but for general recursive data types.

\begin{theorem}[Uniform Kruskal-Friedman Theorem for~$\nu$]\label{thm:unif-KF}
When $W$ is a normal $\wpo$-dilator, the initial $\nu$-Kruskal fixed point of~$W$ is a well partial order.
\end{theorem}

In addition to proving the theorem, we aim to analyze it from the viewpoint of reverse mathematics. As a precondition, we need to explain how our definitions are to be formalized in that framework. The crucial observation is that any $\po$-dilator is determined by its restriction to the category of finite partial orders. For the case of dilators on linear orders, the analogous result is due to Girard~\cite{girard-pi2}. It~rests on the idea that any comparison $\sigma_0\leq_{W(X)}\sigma_1$ is already determined on the finite suborder~$a:=\bigcup_{i\leq 1}\supp_X(\sigma_i)$. Indeed, the support condition from Definition~\ref{def:po-dilator} allows us to write $\sigma_i=W(\iota)(\sigma_i')$ for the inclusion~$\iota:a\hookrightarrow X$. Since $W(\iota)$ is an embedding, $\sigma_0\leq_{W(X)}\sigma_1$ is equivalent to $\sigma_0'\leq_{W(a)}\sigma_1'$. This argument also reveals why Definition~\ref{def:po-dilator} includes the condition that $\po$-dilators preserve embeddings. In fact, some readers may wonder why quasi embeddings are considered at all. The reason is that a linearization of an order~$Y$ can be represented by a surjective quasi embedding $f:X\to Y$ where $X$ is linear. For the purpose of formalization, the point is that the restrictions of $\po$-dilators to finite orders are sets rather than proper classes. To represent them in reverse mathematics, we presume that these restrictions are countable and indeed coded by subsets of~$\mathbb N$. In~\cite{frw-kruskal} it is carefully shown how all arguments can be carried out on the level of representations. The latter are not made explicit in the present paper, in order to improve readability.

As we work in the framework of reverse mathematics, we shall identify ordinals with countable well orders. The principle of $\Pi^1_1$-recursion along an ordinal~$\nu$ allows us to define hierarchies $\langle Y_\alpha\,|\,\alpha<\nu\rangle$ by recursive clauses
\begin{equation*}
Y_\alpha=\{n\in\mathbb N\,|\,\varphi(n,\alpha,\langle Y_\gamma\,|\,\gamma<\alpha\rangle)\},
\end{equation*}
where $\varphi$ must be a $\Pi^1_1$-formula (i.\,e., of the form $\forall X\!\subseteq\!\mathbb N\,\forall/\exists\, x_1\!\in\!\mathbb N\ldots\forall/\exists\, x_k\!\in\!\mathbb N.\,\theta$ with no quantifiers in~$\theta$). By $\Pi^1_1$-transfinite recursion, we shall mean the statement that $\Pi^1_1$-recursion is available along any (countable) ordinal. For a more detailed discussion of transfinite recursion in reverse mathematics, we refer to~\cite{simpson09}. Finally, we are ready to state our main result. Concerning the following statement~(i), we note that $\rca_0$ proves the existence and uniqueness of initial $\nu$-Kruskal fixed points (but not the fact that they are well partial orders), as shown in the next section. 

\begin{theorem}\label{thm:main}
Over $\mathsf{RCA}_0$ extended by the chain antichain principle ($\mathsf{CAC}$), the following are equivalent for any infinite ordinal~$\nu$:
\begin{enumerate}[label=(\roman*)]
\item the uniform Kruskal-Friedman theorem for~$\nu$,
\item $\Pi^1_1$-recursion along~$\nu$.
\end{enumerate}
\end{theorem}

The restriction to infinite~$\nu$ is inherited from~\cite{FR_Pi11-recursion}, where it saves a tedious case distinction. It seems likely that the present theorem and the result from~\cite{FR_Pi11-recursion} extend to all finite~$\nu$, but this has not been verified. The important case of~$\nu=1$ is the main result of~\cite{frw-kruskal} (with \cite{freund-equivalence,freund-computable} replacing~\cite{FR_Pi11-recursion}). The chain antichain principle in the base theory of Theorem~\ref{thm:main} ensures that different definitions of well partial order are equivalent. Specifically, we adopt the definition that $X$ is a well partial order if any infinite sequence $x_0,x_1,\ldots\subseteq X$ involves an inequality $x_i\leq_X x_j$ for some~$i<j$. In the presence of $\mathsf{CAC}$, one may equivalently ask for a strictly increasing $f:\mathbb N\to\mathbb N$ such that $x_{f(i)}\leq_X x_{f(i+1)}$ holds for all~$i\in\mathbb N$ (see~\cite{cholak-RM-wpo} for a detailed analysis). Note that $\mathsf{CAC}$ follows from Ramsey's theorem for pairs, which is an extremely weak consequence of our statement~(ii). It is somewhat surprising that no proof of $\mathsf{CAC}$ from statement~(i) has been found so far. According to Lemma~\ref{lem:bootstrap-KF-univ} of the present paper, $\mathsf{CAC}$ follows when~(i) is available for all ordinals. This improves the base theory in the following (where finite~$\nu$ may be included due to Corollary~\ref{cor:KF-monotone}).

\begin{corollary}\label{cor:FP-TR}
The following are equivalent over $\mathsf{RCA}_0$:
\begin{enumerate}[label=(\roman*)]
\item the uniform Kruskal-Friedman theorem for all~$\nu$,
\item $\Pi^1_1$-transfinite recursion.
\end{enumerate}
\end{corollary}

In order to obtain another conclusion of interest, we recall that the following statement~(ii) is equivalent to $\Pi^1_1$-recursion along~$\omega$ (see, e.\,g.,~\cite[Section~9]{FR_Pi11-recursion}).

\begin{corollary}
The following are equivalent over $\mathsf{RCA}_0+\mathsf{CAC}$:
\begin{enumerate}[label=(\roman*)]
\item the uniform Kruskal-Friedman theorem for~$\nu=\omega$,
\item every subset of~$\mathbb N$ is contained in a countable $\beta$-model of $\Pi^1_1$-comprehension.
\end{enumerate}
\end{corollary}

In the rest of this introduction, we summarize how Theorem~\ref{thm:main} will be proved. Let us first observe that the functions $\kappa:\nu\times W(X)\to X$ from Definition~\ref{def:nu-FP} are order preserving, if we agree that $(\alpha,\sigma)\leq(\beta,\tau)$ is equivalent to the conjunction of $\alpha\leq\beta$ and $\sigma\leq\tau$. Now consider a dilator~$D$ on linear orders. To make~$\nu\times D(Z)$ linear, we add that $(\alpha,\sigma)\leq(\beta,\tau)$ holds whenever we have $\alpha<\beta$. But then we cannot expect to get an order preserving map $\psi:\nu\times D(Z)\to Z$, as the domain may always have larger order type than the range. One way out is to allow~$\psi$ to be partial. If~$\psi$ is bijective on its domain, we can represent it by its total inverse
\begin{equation*}
\pi:Y\to\nu\times D(Y).
\end{equation*}
In~\cite[Definition~1.4]{FR_Pi11-recursion} we have formulated a condition which ensures that the range of~$\pi$ is large. If $\pi$ satisfies this condition, it is called a $\nu$-collapse of~$D$. The main result of~\cite{FR_Pi11-recursion} shows that $\Pi^1_1$-recursion along~$\nu$ is equivalent to the statement that any dilator~$D$ on (linear) well orders admits a $\nu$-collapse for some well order~$Y$. We will show that this last statement follows from~(ii) of Theorem~\ref{thm:main}. Given a dilator~$D$, one can find a normal \mbox{$\wpo$-}dilator~$W$ that admits a natural family of order reflecting maps $\eta_X:D(X)\to W(X)$, one for each linear order~$X$. This is a non-trivial result established in~\cite{frw-kruskal}. As in the case of partial orders, it is not hard to construct a $\nu$-collapse $\pi:Y\to\nu\times D(Y)$ for some linear order~$Y$. The crucial task is to show that the latter is well founded. Given~(ii) of Theorem~\ref{thm:main}, the $\nu$-Kruskal fixed point~$Z$ of~$W$ is a well partial order. We will construct an order reflecting function $f:Y\to Z$ (i.\,e., a partial linearization of~$Y$ by~$Z$) such that
\begin{equation*}
\begin{tikzcd}[column sep=large]
Y\ar[rr,"f"]\ar[d,swap,"\pi"] & & Z\\
\nu\times D(Y)\ar[r,"\nu\times\eta_Y"] & \nu\times W(Y)\ar[r,"\nu\times W(f)"] & \nu\times W(Z)\ar[u,swap,"\kappa"]
\end{tikzcd}
\end{equation*}
is a commutative diagram. Given that $Z$ is a well partial order, the map~$f$ ensures that~$Y$ is a well order. Now~(i) of Theorem~\ref{thm:main} follows via the cited result from~\cite{FR_Pi11-recursion}. It may be interesting to observe that the given argument involves `abstract versions' of classical results from ordinal analysis: First, the partial inverse~$\psi:\nu\times D(Y)\to Y$ of~$\pi$ is intimately connected to the proof theoretic collapsing functions of W.~Buchholz~\cite{buchholz-new-system}. Secondly, the proof in~\cite{FR_Pi11-recursion} relativizes the ordinal analysis of $\Pi^1_1$-recursion, for which we refer to~\cite{buchholz-local-predicativity,bfps-inductive,jaeger-pohlers82,Rathjen_PhD_2013}. Finally, the embedding $f:Y\to X$ corresponds to linearizations of Friedman's gap condition on trees by traditional ordinal notation systems, as given in~\cite{buchholz_perspicious,simpson85}.

To conclude, we sketch the proof that~(ii) implies~(i) in Theorem~\ref{thm:main}. As we will see, statement~(i) follows from a result in~\cite{kriz-conjecture}, which we will call `K\v{r}\'{i}\v{z}'s minimality principle' (see Section~\ref{sect:kriz-minimal} for the precise statement). This principle is an existence result for minimal bad sequences in the context of gap embeddings. In a posting to the `Foundations of Mathematics' mailing list, Friedman has suggested that K\v{r}\'{i}\v{z}'s proof ``uses something like light faced $\Pi^1_2\text{-}\mathsf{CA}_0$"~\cite{friedman-polemics} (it ostensibly uses an ordinal with cofinality above~$\omega$). We will refine the proof to show that $\Pi^1_1$-recursion implies K\v{r}\'{i}\v{z}'s minimality principle and thus the uniform Kruskal-Friedman theorem. Note that both implications are equivalences, as the circle is closed by Theorem~\ref{thm:main}.

\section{Existence of Kruskal fixed points}\label{sect:fp-exist}

In this section, we prove the existence of initial $\nu$-Kruskal fixed points via an explicit construction. We also provide a criterion that helps to decide whether a given fixed point is initial. Unless otherwise noted, we work over the base \mbox{theory $\rca_0$}. The presentation in this section is similar to the one in \cite{frw-kruskal} and~\cite{freund-kruskal-gap}, where the analogous results for~$\nu=1$ are proved.

We begin with a definition and lemma that will be used to approximate $\nu$-Kruskal fixed points via their finite suborders.

\begin{definition}
Given a $\po$-dilator~$W$ and a finite partial order~$a$, we put
\begin{equation*}
W_0(a):=\{\sigma\in W(a)\,|\,\supp_a(\sigma)=a\}.
\end{equation*}
\end{definition}

Let us point out that the given definition is closely related to the notion of trace, which is due to Girard~\cite{girard-pi2} for linear orders and was transferred to partial orders in~\cite{frw-kruskal}. The following can be seen as a variant of Girard's normal form theorem.

\begin{lemma}\label{lem:normal-form}
For each partial order~$X$ and each $\sigma\in W(X)$, there are unique $a\in[X]^{<\omega}$ and $\sigma_0\in W_0(a)$ with $\sigma=W(\iota)(\sigma_0)$, where $\iota:a\hookrightarrow X$ is the inclusion.
\end{lemma}
\begin{proof}
To obtain uniqueness, note that $\sigma=W(\iota)(\sigma_0)$ with $\sigma_0\in W_0(a)$ entails
\begin{equation*}
\supp_X(\sigma)={\supp_X}\circ W(\iota)(\sigma_0)=[\iota]^{<\omega}\circ\supp_a(\sigma_0)=a,
\end{equation*}
due to the naturality of supports. Once~$a$ is determined, so is $\sigma_0$, as the embedding $W(\iota)$ is injective. For existence, set $a:=\supp_X(\sigma)$ and use the support condition from Definition~\ref{def:po-dilator} to get $\sigma=W(\iota)(\sigma_0)$ for some $\sigma_0\in W(a)$. We have $\sigma_0\in W_0(a)$ by the same computation as in the uniqueness part.
\end{proof}

The following definition yields an initial $\nu$-Kruskal fixed point of~$W$, as we will prove below. It extends the construction for~$\nu=1$ given in~\cite[Definition~3.4]{frw-kruskal}. To allow for applications of $W$, the definition requires that the arising relation is a partial order on certain subsets. In hindsight, this requirement turns out to be redundant: we will see that we get a partial order on the entire fixed point.

\begin{definition}\label{def:TW-nu}
For an ordinal~$\nu$ and a normal~$\po$-dilator~$W$, we use simultaneous recursion to generate a set $\mathcal T W[\nu]$ of terms and a binary relation $\leq_{\mathcal TW[\nu]}$ on it:
\begin{itemize}[itemsep=.5ex]
\item Whenever we have constructed a finite set $a\subseteq\mathcal TW[\nu]$ that is partially ordered by the relation~$\leq_{\mathcal TW[\nu]}$, we add a term $\alpha\star(a,\sigma)\in\mathcal TW[\nu]$ for each ordinal $\alpha<\nu$ and each element $\sigma\in W_0(a)$.
\item The relation $\alpha\star(a,\sigma)\leq_{\mathcal TW[\nu]}\beta\star(b,\tau)$ holds precisely when we have $\alpha\leq\beta$ and one of the following conditions is satisfied:
\begin{enumerate}[label=(\roman*),itemsep=.5ex,topsep=.5ex]
\item the set $a\cup b$ is partially ordered by~$\leq_{\mathcal TW[\nu]}$ and we have
\begin{equation*}
W(\iota_a)(\sigma)\leq_{W(a\cup b)} W(\iota_b)(\tau)
\end{equation*}
for the inclusions $\iota_a:a\hookrightarrow a\cup b$ and $\iota_b:b\hookrightarrow a\cup b$,
\item we have $\alpha\star(a,\sigma)\leq_{\mathcal TW[\nu]} r$ for some~$r\in b$.
\end{enumerate}
\end{itemize}
\end{definition}

More explicitly, the given clauses decide $r\in\mathcal T W[\nu]$ and $s\leq_{\mathcal T W[\nu]}t$ by simultaneous recursion on~$l(r)$ and $l(s)+l(t)$, respectively, for the length function
\begin{equation*}
l:\mathcal T W[\nu]\to\mathbb N\quad\text{with}\quad l(\alpha\star(a,\sigma)):=1+\textstyle\sum_{s\in a}2\cdot l(s).
\end{equation*}
To prove that the relation $\leq_{\mathcal TW}$ is a partial order, we first show how it relates to the height function
\begin{equation*}
h:\mathcal T W[\nu]\to\mathbb N\quad\text{with}\quad h(\alpha\star(a,\sigma)):=\sup\{h(s)+1\,|\,s\in a\}.
\end{equation*}
As indicated in the introduction, the assumption that~$W$ is normal is crucial for the following argument.

\begin{lemma}\label{lem:heights}
From $s\leq_{\mathcal TW[\nu]}t$ we get $h(s)\leq h(t)$.
\end{lemma}
\begin{proof}
We argue by induction on~$l(s)+l(t)$. Write $s=\alpha\star(a,\sigma)$ and $t=\beta\star(b,\tau)$. If $s\leq_{\mathcal TW[\nu]} t$ holds by clause~(ii) from Definition~\ref{def:TW-nu}, we get $h(s)\leq h(r)<h(t)$ by induction. If it holds by clause~(i), then normality as in Definition~\ref{def:normal} yields
\begin{equation*}
a={\supp_{a\cup b}}\circ W(\iota_a)(\sigma)\leqf_{\mathcal TW[\nu]}{\supp_{a\cup b}}\circ W(\iota_b)(\tau)=b.
\end{equation*}
Here the equalities rely on $\sigma\in W_0(a)$ and $\tau\in W_0(b)$ due to the definition of $\mathcal TW[\nu]$. For any $s'\in a$ we thus find a $t'\in b$ with $s'\leq_{\mathcal TW[\nu]}t'$. By induction hypothesis this gives $h(s')\leq h(t')<h(t)$. As $s'\in a$ was arbitrary, we get $h(s)\leq h(t)$.
\end{proof}

The previous lemma ensures that clause~(ii) of Definition~\ref{def:TW-nu} cannot `break' antisymmetry. Indeed, if $s\leq_{\mathcal TW[\nu]}t$ holds by said clause, we get $h(s)<h(t)$ by the proof of the lemma. But then the latter excludes $t\leq_{\mathcal TW[\nu]} s$. Apart from this observation -- which motivates the notion of normality -- we do not provide a proof of the following result: instead we refer to the proof of~\cite[Proposition~3.6]{frw-kruskal}, which treats the case~$\nu=1$ and generalizes without essential changes.

\begin{lemma}
The relation $\leq_{\mathcal TW[\nu]}$ is a partial order on~$\mathcal TW[\nu]$.
\end{lemma}

We now come to a crucial construction, which is justified by Lemma~\ref{lem:normal-form}.

\begin{definition}\label{def:kappa}
Given an ordinal~$\nu$ and a normal $\po$-dilator~$W$, we define
\begin{equation*}
\kappa_W^\nu:\nu\times W(\mathcal TW[\nu])\to\mathcal TW[\nu]
\end{equation*}
by stipulating $\kappa_W^\nu(\alpha,\sigma)=\alpha\star(a,\sigma_0)$ for $\sigma=W(\iota)(\sigma_0)$ with $\sigma_0\in W_0(a)$, where we write $\iota:a\hookrightarrow\mathcal TW[\nu]$ for the inclusion.
\end{definition}

Let us show that our definition yields the desired result:

\begin{proposition}\label{prop:TWnu-fp}
The partial order~$\mathcal TW[\nu]$ and the function $\kappa_W^\nu$ form a $\nu$-Kruskal fixed point of~$W$, whenever the latter is a normal $\po$-dilator.
\end{proposition}
\begin{proof}
First note that $\kappa_W^\nu$ has an obvious inverse and is thus bijective. To establish the equivalence from Definition~\ref{def:nu-FP}, we look at two values
\begin{equation*}
\kappa_W^\nu(\alpha,\sigma)=\alpha\star(a,\sigma_0)\quad\text{and}\quad\kappa_W^\nu(\beta,\tau)=\beta\star(b,\tau_0).
\end{equation*}
Consider the inclusions $\iota_a:a\hookrightarrow a\cup b$ and $\iota_b:b\hookrightarrow a\cup b$ as well as $\iota:a\cup b\hookrightarrow\mathcal TW[\nu]$. According to the definition of~$\kappa_W^\nu$, we have
\begin{equation*}
\sigma=W(\iota\circ\iota_a)(\sigma_0)\quad\text{and}\quad\tau=W(\iota\circ\iota_b)(\tau_0).
\end{equation*}
In view of $\tau_0\in W_0(b)$ we get $\supp_{\mathcal TW[\nu]}(\tau)=b$, as in the proof of Lemma~\ref{lem:normal-form}. Given that $W(\iota)$ is an embedding, the desired equivalence from Definition~\ref{def:nu-FP} amounts to
\begin{multline*}
\alpha\star(a,\sigma_0)\leq_{\mathcal TW[\nu]}\beta\star(b,\tau_0)\quad\Leftrightarrow{}\\
\alpha\leq\beta\text{ and }\left(\,W(\iota_a)(\sigma_0)\leq_{W(a\cup b)}W(\iota_b)(\tau_0)\text{ or }\kappa(\alpha,\sigma)\leqf_{\mathcal TW[\nu]}b\,\right).
\end{multline*}
This coincides with the recursive characterization of $\leq_{\mathcal TW[\nu]}$ in Definition~\ref{def:TW-nu}.
\end{proof}

In the following result, the equivalence between~(i) and~(ii) is the criterion for initial fixed points that was promised at the beginning of this section. In particular, this criterion can be used to confirm that the uniform Kruskal-Friedman theorem entails the original version for trees, as noted in the introduction. Statement~(iii) will allow us to relate Kruskal fixed points for different ordinals. A similar criterion for the case~$\nu=1$ is provided by~\cite[Theorem~3.5]{freund-kruskal-gap}

\begin{theorem}\label{thm:initial-criterion}
Assume that $X$ and $\kappa:\nu\times W(X)\to X$ form a $\nu$-Kruskal fixed point of a normal $\po$-dilator~$W$. The following are equivalent:
\begin{enumerate}[label=(\roman*)]
\item The $\nu$-Kruskal fixed point $(X,\kappa)$ is initial.
\item There is a function $h:X\to\mathbb N$ such that we have
\begin{equation*}
h(x)<h(\kappa(\alpha,\sigma))\quad\text{whenever}\quad x\in\supp_X(\sigma).
\end{equation*}
\item For any embedding $I:\nu\to\mu$ and any $\mu$-Kruskal fixed point $(Y,\pi)$ of~$W$, there is a unique quasi embedding $f:X\to Y$ such that the diagram
\begin{equation*}
\begin{tikzcd}
\nu\times W(X)\arrow[r,"\kappa"]\arrow[d,swap,"I\times W(f)"] & X\arrow[d,"f"]\\
\mu\times W(Y)\arrow[r,"\pi"] & Y
\end{tikzcd}
\end{equation*}
commutes, where $I\times W(f)$ maps $(\alpha,\sigma)$ to $(I(\alpha),W(f)(\sigma))$.
\end{enumerate}
Furthermore, the unique~$f$ in~(iii) is always an embedding.
\end{theorem}
\begin{proof}
First, a glance at Definition~\ref{def:nu-FP} reveals that (iii) entails~(i) as the special case where~$I$ is the identity on~$\nu=\mu$. Next, we assume~(i) and derive~(ii). In view of Proposition~\ref{prop:TWnu-fp}, we get a quasi embedding $f:X\to\mathcal TW[\nu]$ with
\begin{equation*}
f\circ\kappa=\kappa_W^\nu\circ(\nu\times W(f)).
\end{equation*}
Let $h:\mathcal TW[\nu]\to\mathbb N$ be the function defined before Lemma~\ref{lem:heights}. We show that~(ii) holds with $h\circ f:X\to\mathbb N$ at the place of~$h$. Given any $\alpha<\nu$ and $\sigma\in W(X)$, invoke Lemma~\ref{lem:normal-form} to write $W(f)(\sigma)=W(\iota)(\sigma_0)$ with $\iota:a\hookrightarrow\mathcal TW[\nu]$ and $\sigma_0\in W_0(a)$. For an arbitrary element $x\in\supp_X(\sigma)$, we get
\begin{equation*}
f(x)\in[f]^{<\omega}\circ\supp_X(\sigma)={\supp_{\mathcal TW[\nu]}}\circ W(f)(\sigma)=a,
\end{equation*}
where the last equality is established as in the proof of Lemma~\ref{lem:normal-form}. Considering~the definition of~$h$, we can conclude
\begin{equation*}
h\circ f(x)<h(\alpha\star(a,\sigma_0))=h\left(\kappa_W^\nu(\alpha,W(f)(\sigma))\right)=h\circ f(\kappa(\alpha,\sigma)),
\end{equation*}
as needed to establish~(ii) for~$h\circ f$. Finally, we assume~(ii) and deduce~(iii). To see that there is at most one~$f$ with the required property, note that an arbitrary element of~$X$ can be uniquely written as~$\kappa(\alpha,\sigma)$, since $\kappa$ is bijective by Definition~\ref{def:nu-FP}. Commutativity of the diagram in~(iii) amounts to
\begin{equation}\tag{$\star$}\label{eq:initial}
f\circ\kappa(\alpha,\sigma)=\pi(I(\alpha),W(f\circ\iota)(\sigma_0))\quad\text{for}\quad \sigma=W(\iota)(\sigma_0),
\end{equation}
where we take $\iota:a\hookrightarrow X$ and $\sigma_0\in W_0(a)$ to be determined by Lemma~\ref{lem:normal-form}. As~in the proof of the latter, we learn that the domain~$a$ of $f\circ\iota$ is equal to $\supp_X(\sigma)$. In view of~(\ref{eq:initial}), we can thus invoke induction over~$h(x)$ to see that $f(x)$ is uniquely determined. For existence, the idea is to read~(\ref{eq:initial}) as a recursive definition of~$f$. To ensure that the application of~$W$ is explained, we must simultaneously verify that $f\circ\iota$ is a (quasi) embedding. In order to make this precise, we use recursion over~$h$ to define a length function
\begin{equation*}
l:X\to\mathbb N\quad\text{with}\quad l(\kappa(\alpha,\sigma))=1+\textstyle\sum_{x\,\in\,\supp_X(\sigma)}2\cdot l(x).
\end{equation*}
One can now use~(\ref{eq:initial}) to define $f(x)$ by recursion on $l(x)$, while showing
\begin{equation*}
x_0\leq_X x_1\quad\Leftrightarrow\quad f(x_0)\leq_Y f(x_1)
\end{equation*}
by a simultaneous induction on~$l(x_0)+l(x_1)<l(x)$. This last induction is based on the fact that $\kappa$ and $\pi$ satisfy the same equivalence from Definition~\ref{def:nu-FP}. For details we refer to the proof of the case~$\nu=1$ in~\cite[Theorem~3.5]{freund-kruskal-gap}.
\end{proof}

With respect to the constructions from Definitions~\ref{def:TW-nu} and~\ref{def:kappa}, we can now strengthen Proposition~\ref{prop:TWnu-fp} as follows.

\begin{corollary}\label{cor:Kruskal-fp-exist}
The $\nu$-Kruskal fixed point $(\mathcal TW[\nu],\kappa_W^\nu)$ of $W$ is initial, for any normal $\po$-dilator~$W$.
\end{corollary}
\begin{proof}
Statement~(ii) of Theorem~\ref{thm:initial-criterion} holds for the function $h:\mathcal T W[\nu]\to\mathbb N$ defined before Lemma~\ref{lem:heights}, as shown in the proof of the theorem.
\end{proof}

As pointed out in the introduction, all initial $\nu$-Kruskal fixed points of~$W$ are isomorphic, by a general argument about universal properties: Given two initial fixed points $(X,\kappa)$ and $(Y,\pi)$, we get quasi embeddings $f$ and $g$ such that
\begin{equation*}
\begin{tikzcd}
\nu\times W(X)\arrow[r,"\kappa"]\arrow[d,swap,"\nu\times W(f)"]\arrow[dd,swap,bend right=60,start anchor=west,end anchor=west,"\nu\times W(g\circ f)"] & X\arrow[d,"f"]\arrow[dd,bend left=60,start anchor=east,end anchor=east,"g\circ f"]\\
\nu\times W(Y)\arrow[r,"\pi"]\arrow[d,swap,"\nu\times W(g)"] & Y\arrow[d,"g"]\\
\nu\times W(X)\arrow[r,"\kappa"] & X
\end{tikzcd}
\end{equation*}
is a commutative diagram. Now there can only be one quasi embedding $g\circ f$ such that the outer rectangle commutes, by the uniqueness condition in Definition~\ref{def:nu-FP}. Hence $g\circ f$ must be the identity on~$X$. In the same way we see that $f\circ g$ is the identity on~$Y$, which makes $f$ an isomorphism with inverse~$g$. Having shown existence and essential uniqueness, we may speak of `the' initial $\nu$-Kruskal fixed point of a normal $\po$-dilator~$W$. Furthermore, we can use $\mathcal TW[\nu]$ as notation for `this' fixed point, without committing to the specific construction from Definition~\ref{def:TW-nu}. The formulation of Theorem~\ref{thm:main} from the introduction is now fully explained and justified. We conclude this section with the following observation.

\begin{corollary}\label{cor:KF-monotone}
For ordinals $\nu<\mu$ and any normal $\po$-dilator~$W$, we have an order embedding of $\mathcal TW[\nu]$ into $\mathcal TW[\mu]$.
\end{corollary}
\begin{proof}
The assumption $\nu<\mu$ is witnessed by an embedding~$I:\nu\to\mu$. We can conclude via statement~(iii) of Theorem~\ref{thm:initial-criterion} (which applies due to Corollary~\ref{cor:Kruskal-fp-exist}).
\end{proof}

It follows that the uniform Kruskal-Friedman theorem holds for all finite~$\nu<\omega$ if it holds for some infinite~$\nu$, as needed to infer Corollary~\ref{cor:FP-TR} from Theorem~\ref{thm:main}.

\section{K\v{r}\'{i}\v{z}'s minimality principle}\label{sect:kriz-minimal}

In this section, we derive the uniform Kruskal-Friedman theorem from a minimal bad sequence principle for structures with gap condition, which is due to~K\v{r}\'{i}\v{z}~\cite{kriz-conjecture}. We also show that this principle can be derived from $\Pi^1_1$-recursion, by presenting a modified version of K\v{r}\'{i}\v{z}'s original argument.

The following coincides with~\cite[Definition~1.4]{kriz-conjecture}, except that clause~(v) has been added (cf.~the proof of Theorem~\ref{thm:kriz-minimality} and the remark before Theorem~\ref{thm:Pi11-rec-to-Kriz}).

\begin{definition}\label{def:gap-order}
For an ordinal~$\nu$, a $\nu$-gap order consists of a quasiorder~$(X,\leq_X)$, a function $q:X\to\nu$ and a binary relation~$\ll$ on~$X$, such that the following holds:
\begin{enumerate}[label=(\roman*),itemsep=.5ex,topsep=.5ex]
\item if we have $s\leq_X t$, then we must also have $q(s)\leq q(t)$,
\item from $r\leq_X s\ll t$ and $q(r)\leq q(t)$, we can infer $r\leq_X t$,
\item given $s_0\ll\ldots\ll s_n$ with $\min_{i\leq n}q(s_i)\in\{q(s_0),q(s_n)\}$, we get $s_0\ll s_n$,
\item there is no infinite sequence $f:\mathbb N\to X$ with $f(i+1)\ll f(i)$ for all~$i\in\mathbb N$,
\item the set $\{s\in X\,|\,s\ll t\}$ is finite for each~$t\in X$.
\end{enumerate}
We shall write~$\ll^*$ for the reflexive and transitive closure of~$\ll$.
\end{definition}

Let us now show that any initial $\nu$-Kruskal fixed point (see the previous section) can be seen as $\nu$-gap order. With respect to the following definition, we point out that $\kappa:\nu\times W(X)\to X$ is a bijection by Definition~\ref{def:nu-FP}. The recursion is justified due to part~(ii) of Theorem~\ref{thm:initial-criterion} (with $\supp_X:W(X)\to[X]^{<\omega}$ as in Definition~\ref{def:po-dilator}).

\begin{definition}\label{def:fp-gap-order}
We consider a normal~$\po$-dilator~$W$ and its initial $\nu$-Kruskal fixed point~$(X,\kappa)$. Let the function $q:X\to\nu$ be given by $q(\kappa(\alpha,\sigma)):=\alpha$. For $\gamma<\nu$ we define $K^W_\gamma:X\to[X]^{<\omega}$ and $K_\gamma:W(X)\to[X]^{<\omega}$ by the recursive clauses
\begin{align*}
K^W_\gamma(\kappa(\alpha,\sigma))&:=\begin{cases}
\{\kappa(\alpha,\sigma)\}\cup K_\gamma(\sigma) & \text{if }\alpha\geq\gamma,\\
\emptyset & \text{otherwise},
\end{cases}\\
K_\gamma(\sigma)&:=\bigcup\{K^W_\gamma(r)\,|\,r\in\supp_X(\sigma)\}.
\end{align*}
To obtain a binary relation~$\ll$ on~$X$, we now declare that $\kappa(\alpha,\sigma)\ll\kappa(\beta,\tau)$ is equivalent to $\kappa(\alpha,\sigma)\in K_\gamma(\tau)$ for $\gamma=\min\{\alpha,\beta\}$.
\end{definition}

Intuitively, we have $s\ll t$ when $s$ is a proper subtree of~$t$ and the minimal label on the path between the roots of~$s$ and~$t$ occurs at one of these roots. The following extends~\cite[Lemma~1.6]{kriz-conjecture} from trees to general recursive data types.

\begin{lemma}\label{lem:fp-gap-order}
The previous definition yields a $\nu$-gap order $(X,\leq_X,q,\ll)$.
\end{lemma}
\begin{proof}
Condition~(i) of Definition~\ref{def:gap-order} follows from the equivalence in Definition~\ref{def:nu-FP}. To show condition~(ii), we note that $r\leq_Xs\ll t$ and $q(r)\leq q(t)$ entail $q(r)\leq q(s)$ and hence $s\in K^W_\gamma(t)$ with $q(r)\leq\gamma$. So it suffices to prove
\begin{equation*}
r\leq_X s\in K^W_\gamma(t)\quad\text{and}\quad q(r)\leq\gamma\qquad\Rightarrow\qquad r\leq_X t.
\end{equation*}
We employ induction on~$h(t)$, for $h:X\to\mathbb N$ as in part~(ii) of Theorem~\ref{thm:initial-criterion}. In the non-trivial case of $s\neq t$, write $t=\kappa(q(t),\tau)$ to get $s\in K^W_\gamma(t')$ with~\mbox{$t'\in\supp_X(\tau)$}. We obtain $r\leq_X t'$ by induction hypothesis. Given $q(r)\leq q(t)$, the equivalence from Definition~\ref{def:nu-FP} yields $r\leq_X t$, as required. As preparation for the other conditions, one uses induction on~$h(t)$ to verify
\begin{equation*}
s\in K^W_\gamma(t)\quad\Rightarrow\quad h(s)\leq h(t)\text{ and } K^W_\gamma(s)\subseteq K^W_\gamma(t)\subseteq K^W_\delta(t)\text{ for }\gamma\geq\delta.
\end{equation*}
One can infer that $s\ll t$ entails $h(s)<h(t)$, which yields condition~(iv). Furthermore, it follows that any $s\ll t=\kappa(q(t),\tau)$ is contained in the finite set~$K_0(\tau)$, so that condition~(v) is satisfied. To derive condition~(iii), we write $s_i=\kappa(q(s_i),\sigma_i)$. Given $\min_{i\leq n}q(s_i)\in\{q(s_0),q(s_n)\}$, we have
\begin{equation*}
\gamma(i):=\min\{q(s_i),q(s_{i+1})\}\geq\min\{q(s_0),q(s_n)\}=:\gamma
\end{equation*}
for all~$i<n$. Assuming $s_0\ll\ldots\ll s_n$, we also get
\begin{equation*}
s_i\in K_{\gamma(i)}(\sigma_{i+1})\subseteq K^W_{\gamma(i)}(s_{i+1})\subseteq K^W_\gamma(s_{i+1}).
\end{equation*}
This entails $K^W_\gamma(s_i)\subseteq K^W_\gamma(s_{i+1})$ and then $s_0\in K^W_\gamma(s_n)$. As we have $h(s_0)<h(s_n)$ and hence $s_0\neq s_n$, we even get $s_0\in K_\gamma(\sigma_n)$, which amounts to $s_0\ll s_n$. 
\end{proof}

Let us agree to write $[N]^\omega$ for the set of infinite subsets of a given set~$N\subseteq\mathbb N$. To formulate the minimality principle of K\v{r}\'{i}\v{z}, we need the following terminology.

\begin{definition}
Consider a $\nu$-gap order $(X,\leq_X,q,\ll)$. A sequence $f:M\to X$ with $M\in[\mathbb N]^\omega$ is called bad if there are no $i,j\in M$ with $i<j$ and $f(i)\leq_X f(j)$. It is called regular if we have $q(f(i))\leq q(f(j))$ for all~$i<j$ in~$M$.
\end{definition}

The following coincides with \cite[Theorem~1.5]{kriz-conjecture}, except that our Definition~\ref{def:gap-order} is more restrictive than~\cite[Definition~1.4]{kriz-conjecture}. This means that the present statement is somewhat weaker than K\v{r}\'{i}\v{z}'s original result.

\begin{theorem}[K\v{r}\'{i}\v{z}'s minimality principle for~$\nu$]\label{thm:kriz-minimality}
Assume that $(X,\leq_X,q,\ll)$ is a $\nu$-gap order. Given any bad sequence $f:\mathbb N\to X$, one can find an $M\in[\mathbb N]^\omega$ and a regular bad sequence $g:M\to X$ such that
\begin{enumerate}[label=(\roman*)]
\item we have $g(i)\ll^*f(i)$ for all~$i\in M$,
\item there is no bad $h:N\to X$ with $N\in[M]^\omega$ and $h(i)\ll g(i)$ for all~$i\in N$.
\end{enumerate}
\end{theorem}

Below, we show that K\v{r}\'{i}\v{z}'s minimality principle can be proved by $\Pi^1_1$-recursion. Together with the following result, this entails that (ii) implies~(i) in Theorem~\ref{thm:main}. We note that the following extends \cite[Theorem~1.7]{kriz-conjecture} from trees to general recursive data types. A related generalization is given in~\cite[Section~3]{kriz-conjecture}.

\begin{theorem}\label{thm:Kriz-to-KF}
The uniform Kruskal-Friedman theorem for~$\nu$ follows from K\v{r}\'{i}\v{z}'s minimality principle for~$\nu$, for any ordinal~$\nu$ and over~$\rca_0$.
\end{theorem}
\begin{proof}
Given a normal $\wpo$-dilator~$W$, we consider its initial $\nu$-Kruskal fixed point $(X,\kappa)$ and the $\nu$-gap order~$(X,\leq_X,q,\ll)$ that arises from Lemma~\ref{lem:fp-gap-order}. To establish the uniform Kruskal-Friedman theorem, we must show that~$X$ is a well partial order. Aiming at a contradiction, we assume that there is a bad sequence~$f:\mathbb N\to X$. By K\v{r}\'{i}\v{z}'s minimality principle, we get a regular bad $g:M\to X$ as in Theorem~\ref{thm:kriz-minimality}. Let us write $g(i)=\kappa(q(g(i)),\sigma_i)$ in order to define
\begin{equation*}
Y:=\textstyle\bigcup_{i\in M}\supp_X(\sigma_i)\subseteq X.
\end{equation*}
We show that~$Y$ is a well partial order. If not, we find a bad sequence $h:N\to Y$ with $N\in[M]^\omega$ and $h(i)\in\supp_X(\sigma_i)$, as the supports are finite by Definition~\ref{def:po-dilator}. In view of Definition~\ref{def:fp-gap-order}, we put $\gamma:=\min\{q(g(i)),q(h(i))\}$ to get
\begin{equation*}
h(i)\in K^W_\gamma(h(i))\subseteq K_\gamma(\sigma_i)
\end{equation*}
and hence~$h(i)\ll g(i)$. We have reached a contradiction with~(ii) from Theorem~\ref{thm:kriz-minimality}. Let us now consider the inclusion~$\iota:Y\hookrightarrow X$. We may write $\sigma_i=W(\iota)(\tau_i)$ with~$\tau_i\in W(Y)$, due to the support condition from Definition~\ref{def:po-dilator}. Given that $Y$ is a well partial order, so is~$W(Y)$, since $W$ is a $\wpo$-dilator. We thus get $\tau_i\leq_{W(Y)}\tau_j$ for some $i<j$ in~$M$. Since $W$ preserves embeddings (see again Definition~\ref{def:po-dilator}), we can conclude~$\sigma_i\leq_{W(X)}\sigma_j$. The latter entails $g(i)\leq_X g(j)$ by the equivalence from Definition~\ref{def:nu-FP}, as $g$ is regular. This is in contradiction with $g$ being bad.
\end{proof}

We now present K\v{r}\'{i}\v{z}'s proof~\cite{kriz-conjecture} of his minimality principle with some small but crucial modifications (which will be pointed out). Due to these modifications, the principle of $\Pi^1_1$-recursion suffices to carry out the proof, as we shall argue below.

\begin{proof}[Proof of Theorem~\ref{thm:kriz-minimality}]
Let us begin with some preparations. For each ordinal~$\beta$, we truncate the given $q:X\to\nu$ into a function
\begin{equation*}
q_\beta:X\to\nu\quad\text{with}\quad q_\beta(t):=\begin{cases}
0 & \text{if }q(t)<\beta,\\
q(t) & \text{otherwise}.
\end{cases}
\end{equation*}
Conditions~(iv) and~(v) of Definition~\ref{def:gap-order} yield a rank function $r:X\to\omega$ such that $s\ll t$ entails $r(s)<r(t)$. While condition~(v) will be crucial later, it is not essential at this point, as one could also work with infinite ordinal ranks. We now consider
\begin{equation*}
p_\beta:X\to\omega\cdot\nu\quad\text{with}\quad p_\beta(t):=\omega\cdot q_\beta(t)+r(t).
\end{equation*}
At stage~$\beta+1$ of the recursive construction below, we minimize with respect to~$p_\beta$ where K\v{r}\'{i}\v{z} uses~$q$. We will see that this forces our construction to close after~$\nu+1$ stages, while K\v{r}\'{i}\v{z} gives no information on the closure ordinal, except that it must be countable. We now make precise what we mean by minimizing. For $k\in\mathbb N$ and $g:M\to X$ with $M\subseteq\mathbb N$, we put~$[k]:=\{0,\ldots,k-1\}$ and define $g[k]:M\cap[k]\to X$ as the restriction of~$g$. Given $g_i:M_i\to X$ with $M_i\in[\mathbb N]^\omega$ for $i\in\{0,1\}$, we write $g_0\ang_\beta g_1$ if we have $M_0\subseteq M_1$ and there is a $k\in M_1$ with $g_0[k]=g_1[k]$ and
\begin{equation*}
p_\beta\big(g_0(\min\{k'\in M_0\,|\,k'\geq k\})\big)<p_\beta\big(g_1(k)\big).
\end{equation*}
There is a close connection with the relation~$\ang$ from~\cite[Definition~2.1]{kriz-conjecture} and the minimal bad sequences of Nash-Williams~\cite{nash-williams63}. For the purpose of this proof, we declare that a tree is a non-empty set~$T$ of functions $\sigma:a\to X$ with finite domain~$a\subseteq\mathbb N$, such that $\sigma\in T$ entails $\sigma[k]\in T$. By a branch of~$T$ we mean an infinite function~$g:M\to X$ with $g[k]\in T$ for all~$k\in\mathbb N$. As usual, we write $[T]$ for the set of branches of~$T$ (even though this overloads the notation~$[\cdot]$). Given a non-empty set~$\mathcal S$ of infinite functions~$g:M\to X$, we define a tree $T(\mathcal S)$ by stipulating
\begin{equation*}
\sigma\in T(\mathcal S)\quad:\Leftrightarrow\quad \sigma=g[k]\text{ for some }g\in\mathcal S\text{ and }k\in\mathbb N.
\end{equation*}
We shall use the following variant of~\cite[Lemma~2.2]{kriz-conjecture}, which is standard but will also be verified in the proof of Theorem~\ref{thm:Pi11-rec-to-Kriz} below:
\begin{equation}\label{eq:Lem-2-2}\tag{$\dag$}
\parbox{.9\textwidth}{Consider a non-empty set $\mathcal S$ of infinite sequences~$g:M\to X$ with $M\subseteq\mathbb N$. Assume that we have $[T(\mathcal S)]\subseteq\mathcal S$. For any ordinal~$\beta$, we then have an~$h\in\mathcal S$ such that $g\ang_\beta h$ fails for all~$g\in\mathcal S$.}
\end{equation}
Intuitively speaking, the assumption $[T(\mathcal S)]\subseteq\mathcal S$ says that $\mathcal S$ is a property of infinite sequences that can be refuted on finite initial segments.

Aiming at a contradiction, we now assume that Theorem~\ref{thm:kriz-minimality} fails for a given bad~$f:\mathbb N\to X$. The sequence~$f$ itself witnesses that there is a bad $g_0:M_0\to X$ such that $g_0(i)\ll^*f(i)$ holds for all~$i\in M_0\in[\mathbb N]^\omega$. By~(\ref{eq:Lem-2-2}) we may assume that~$g_0$ is minimal, in the sense that $g\ang_0 g_0$ fails for any bad $g:M\to X$ with $g(i)\ll^*f(i)$ for~$i\in M\in[\mathbb N]^\omega$. Let us observe that this minimal~$g_0$ is regular: If not, then we have $q(g_0(k'))<q(g_0(k))$ for some $k<k'$ in~$M_0$. Now let
\begin{equation*}
g:(M_0\cap[k])\cup(M_0\backslash[k'])\to X
\end{equation*}
be the restriction of~$g_0$ to the indicated domain. Given that $q_0$ coincides with $q$, we get $p_0(g(k'))<p_0(g_0(k))$ and hence $g\ang_0 g_0$, which contradicts the minimality~of~$g_0$. Starting with $k_0:=\min(M_0)$ and $g_0:M_0\to X$ as specified, we shall use recursion on the ordinal~$\alpha$ to construct functions $g_\alpha:M_\alpha\to X$ and elements $k_\alpha\in M_\alpha$ that validate the following properties:
\begin{enumerate}[label=(\roman*.$\alpha$),itemsep=.5ex,topsep=.5ex]
\item we have $M_\alpha\in[\mathbb N]^\omega$, and $g_\alpha:M_\alpha\to X$ is regular and bad,
\item for any $\beta<\alpha$ and all $i\in M_\alpha$, we have $i\in M_\beta$ and $q(g_\alpha(i))\geq q(g_\beta(i))$, as~well as $g_\alpha(i)\ll g_\beta(i)$ or $g_\alpha(i)=g_\beta(i)$,
\item if $\alpha$ is a limit, any $i\in M_\alpha=\bigcap_{\beta<\alpha}M_\beta$ admits a $\beta<\alpha$ with $g_\alpha(i)=g_\beta(i)$,
\item if $\alpha=\beta+1$, then we have $g_\alpha(i)=g_\beta(i)$ for all $i\in M_\alpha\cap[k_\alpha]=M_\beta\cap[k_\alpha]$, as well as $g_\alpha(i)\ll g_\beta(i)$ for all $i\in M_\alpha\backslash[k_\alpha]$,
\item if $\alpha=\beta+1$ and $g'_\alpha:M'_\alpha\to X$ is the restriction of~$g_\alpha$ to $M'_\alpha:=M_\alpha\backslash[k_\alpha]$, then $g\ang_\beta g'_\alpha$ fails for any bad sequence $g:M\to X$ with $q(g(i))\geq q(g_\beta(i))$ and $g(i)\ll g_\beta(i)$ for all $i\in M\in[M'_\alpha]^\omega$.
\end{enumerate}
We note that $k_\alpha$ is only relevant when~$\alpha$ is a successor. In this case, (v.$\alpha$) asserts that $g_\alpha'$ is minimally bad with the properties in (ii.$\alpha$) and (iv.$\alpha$). All conditions are valid (and mostly void) when~$\alpha=0$. For the recursion step towards~$\alpha>0$, assume that $g_\gamma:M_\gamma\to X$ and $k_\gamma\in M_\gamma$ are defined and that (i.$\gamma$)~to~(v.$\gamma$) hold for $\gamma<\alpha$.

Let us first consider the case where~$\alpha$ is a limit. The crucial task is to show that the set $M_\alpha:=\bigcap_{\beta<\alpha}M_\beta$ is infinite. As in the proof of K\v{r}\'{i}\v{z}, we map $\beta<\alpha$ to
\begin{equation*}
k(\beta):=\min\{k_{\gamma+1}\,|\,\beta<\gamma<\alpha\}.
\end{equation*}
Note that $\alpha\ni\beta\mapsto k(\beta)\in\mathbb N$ is non-decreasing. Let us show that it is unbounded: If~not, pick a $\beta_0<\alpha$ such that $k:=k(\beta_0)=k(\beta)$ holds for $\beta_0\leq\beta<\alpha$. We then find a sequence $\beta_0<\gamma(0)<\gamma(1)<\ldots<\alpha$ with $k_{\gamma(i)+1}=k$ for all~$i\in\mathbb N$. Due to conditions~(iv) and~(ii) for $\gamma(i+1)+1$, we obtain
\begin{equation*}
g_{\gamma(i+1)+1}(k)\ll g_{\gamma(i+1)}(k)\ll^* g_{\gamma(i)+1}(k),
\end{equation*}
which contradicts the well foundedness of $\ll$, i.\,e., property~(iv) from Definition~\ref{def:gap-order}. For any $\beta<\alpha$, we now show $k(\beta)\in M_\gamma$ by induction on~$\gamma<\alpha$. When~$\gamma\leq\beta+1$, it suffices to observe that there is a $\delta>\beta$ with $k(\beta)=k_{\delta+1}\in M_{\delta+1}\subseteq M_\gamma$, where the inclusion relies on (ii.$\delta+1$). The induction step towards a limit~$\gamma$ is immediate by condition~(iii.$\gamma$). Finally, for $\gamma=\delta+1>\beta+1$ we have $k(\beta)\leq k_\gamma$ by definition. If this is an equality, the claim is immediate. Otherwise, we get
\begin{equation*}
k(\beta)\in M_\delta\cap[k_\gamma]=M_\gamma\cap[k_\gamma]
\end{equation*}
by the induction hypothesis and (iv.$\gamma$). We can now conclude that
\begin{equation*}
M_\alpha=\textstyle\bigcap_{\gamma<\alpha}M_\gamma\supseteq\{k(\beta)\,|\,\beta<\alpha\}
\end{equation*}
is infinite, as desired. By condition~(ii) and the well foundedness of~$\ll$, each~$i\in M_\alpha$ admits a $\beta(i)<\alpha$ such that $g_\beta(i)=g_{\beta(i)}(i)$ holds for $\beta(i)\leq\beta<\alpha$. To complete the limit case of the recursion step, we define $g_\alpha:M_\alpha\to X$ by $g_\alpha(i):=g_{\beta(i)}(i)$. It is straightforward to check that conditions (i.$\alpha$) to (v.$\alpha$) are satisfied.

Let us continue with the recursion step towards a successor ordinal~$\alpha=\beta+1$. As preparation, we consider an arbitrary bad $h:M\to X$ such that $h(i)\ll g_\beta(i)$ holds for all~$i\in M\in[M_\beta]^\omega$. By induction over~$\gamma\leq\beta$, we show that we have
\begin{equation*}
q(h(k))>q(g_\gamma(k))\quad\text{for all}\quad k\in M.
\end{equation*}
The induction step towards a limit~$\gamma$ is immediate by condition~(iii.$\gamma$). To establish the claim for~$\gamma=0$ and arbitrary~$k\in M$, we consider the sequence
\begin{equation*}
g:(M_0\cap[k])\cup(M\backslash[k])\to X\quad\text{with}\quad g(i):=\begin{cases}
g_0(i) & \text{if } i\in M_0\cap[k],\\
h(i) & \text{if }i\in M\backslash[k].
\end{cases}
\end{equation*}
For $j\in M\backslash[k]$ we can invoke condition (ii.$\beta$) to get $g_\beta(j)=g_0(j)$ or $g_\beta(j)\ll g_0(j)$ and $q(g_\beta(j))\geq q(g_0(j))$. In either case we obtain $h(j)\ll g_0(j)$, due to property~(iii) from Definition~\ref{def:gap-order}. It follows that $g$ is bad: if we had $g(i)\leq_X g(j)$ for $i\in M_0\cap[k]$ and $j\in M\backslash[k]$, we could conclude $g_0(i)\leq_X g_0(j)$ by (ii) of Definition~\ref{def:gap-order}, as $g_0$ is regular. Now $g\ang_0 g_0$ must fail by the choice of~$g_0$. We thus get
\begin{equation*}
\omega\cdot q_0(h(k))+r(h(k))=p_0(g(k))\geq p_0(g_0(k))=\omega\cdot q_0(g_0(k))+r(g_0(k)).
\end{equation*}
Since $h(k)\ll g_0(k)$ entails $r(h(k))<r(g_0(k))$, this yields $q_0(h(k))>q_0(g_0(k))$ and hence $q(h(k))>q(g_0(k))$. In the induction step towards a successor~$\gamma=\delta+1$, we first assume that we have $k\in M\cap[k_\gamma]$. By condition~(iv.$\gamma$) we get $g_\gamma(k)=g_\delta(k)$, so that it suffices to invoke the induction hypothesis. Now assume $k\in M\backslash[k_\gamma]$ and define $g'_\gamma:M'_\gamma\to X$ as in condition~(v.$\gamma$). We observe $M\backslash[k]\subseteq M'_\gamma$ and consider
\begin{equation*}
g':(M'_\gamma\cap[k])\cup(M\backslash[k])\to X\quad\text{with}\quad g'(i):=\begin{cases}
g'_\gamma(i) & \text{if }i\in M'_\gamma\cap[k],\\
h(i) & \text{if } i\in M\backslash[k].
\end{cases}
\end{equation*}
As in the case of $\gamma=0$, one can show that $g'$ is bad. Using the induction hypothesis and conditions (ii.$\gamma$) and (iv.$\gamma$) with $\gamma=\delta+1$, one obtains $q(g'(i))\geq q(g_\delta(i))$ and $g'(i)\ll g_\delta(i)$ for any $i$ in the domain of~$g'$. Invoking the minimality condition~(v.$\gamma$), we can infer $p_\delta(h(k))\geq p_\delta(g_\gamma(k))$. As before, we combine this with $h(k)\ll g_\gamma(k)$ to get $q_\delta(h(k))>q_\delta(g_\gamma(k))$ and then $q(h(k))>q(g_\gamma(k))$, as needed.

After these preparations, we carry out the recursion step towards~$\alpha=\beta+1$. We shall first describe the construction of a sequence $g_\alpha:M_\alpha\to X$ relative to a given set~$N_\beta\in[M_\beta]^\omega$. Afterwards, we will show that conditions (i.$\alpha$) to (v.$\alpha$) hold for a suitable choice of~$N_\beta$, which depends on the form of~$\beta$. As the first step in the construction of $g_\alpha$, we declare that $g^\circ_\beta:N_\beta\to X$ is the restriction of~$g_\beta$. Condition~(i.$\beta$) ensures that $g^\circ_\beta$ is regular and bad. By (ii.$\beta$) and the choice of~$g_0$, we have $g^\circ_\beta(i)\ll^*f(i)$ for all~$i\in N_\beta$. Due to the assumption that Theorem~\ref{thm:kriz-minimality} fails for~$f$, condition~(ii) of the theorem must be false for $g=g^\circ_\beta$. Hence we get a bad sequence $h:N\to X$ with $h(i)\ll g_\beta(i)$ for all $i\in N\in[N_\beta]^\omega$. The preparatory claim above yields $q(h(i))\geq q(g_\beta(i))$. By the principle~(\ref{eq:Lem-2-2}), we may assume that~$h$ is minimal with the given properties, in the sense that (v.$\alpha$) holds with $h:N\to X$ at the place of~$g'_\alpha:M'_\alpha\to X$. For $k_\alpha:=\min(N)$ and $M_\alpha:=(M_\beta\cap[k_\alpha])\cup N$, we now consider the sequence
\begin{equation*}
g_\alpha:M_\alpha\to X\quad\text{with}\quad g_\alpha(i):=\begin{cases}
g_\beta(i) & \text{if }i\in M_\beta\cap[k_\alpha],\\
h(i) & \text{if }i\in N.
\end{cases}
\end{equation*}
This does indeed yield $M'_\alpha=M_\alpha\backslash[k_\alpha]=N$ and $g'_\alpha=h$. So (iv.$\alpha$) and~(v.$\alpha$) hold by construction, while (iii.$\alpha$) is void for a successor. Invoking~(iii) of Definition~\ref{def:gap-order}, one readily derives (ii.$\alpha$) from~(ii.$\beta$). If we had $g_\beta(i)\leq_X h(j)$ with $i\in M_\beta\cap[k_\alpha]$ and~$j\in N$, we would get $g_\beta(i)\leq_X g_\beta(j)$ by~(iii) of Definition~\ref{def:gap-order}, as $g_\beta$ is regular due to~(i.$\beta$). So the sequence~$g_\alpha$ is bad. This reduces~(i.$\alpha$) to the claim that~$h$ is regular. For indices $j<k$ in~$N$, the minimality condition~(v.$\alpha$) entails~$q_\beta(h(j))\leq q_\beta(h(k))$, by the argument that we have used to prove the regularity of~$g_0$. To conclude the recursion step, we now describe a choice of $N_\beta\in[M_\beta]^\omega$ that ensures
\begin{equation}\label{eq_q-beta-to-q}\tag{$\ddag$}
q_\beta(h(j))\leq q_\beta(h(k))\quad\Rightarrow\quad q(h(j))\leq q(h(k))\qquad\text{for }j<k\text{ in }N\in[N_\beta]^\omega.
\end{equation}
Let us point out that the previous arguments are essentially due to K\v{r}\'{i}\v{z}~\cite{kriz-conjecture}. The latter, how\-ever, had no need to show~(\ref{eq_q-beta-to-q}), as he worked with~$q$ at the place of~$p_\beta$. This means that the following is a new contribution of the present paper. To~satisfy~(\ref{eq_q-beta-to-q}) for~$\beta=0$, we can simply take $N_0:=M_0$, as $q_0$ and $q$ coincide. When $\beta=\gamma+1$ is a successor itself, we put~$N_\beta:=M'_\beta=M_\beta\backslash[k_\beta]$. First note that~(\ref{eq_q-beta-to-q}) is immediate when we have $q_\beta(h(j))>0$ and hence $q_\beta(h(j))=q(h(j))$. To cover the case where we have $q_\beta(h(j))=0$ and thus $q(h(j))<\beta=\gamma+1$, it suffices to establish~$q(h(k))\geq\gamma$. Aiming at a contradiction, we assume $q(h(k))<\gamma$, which entails $q_\gamma(h(k))=0$. From above we have $h(k)\ll g_\beta(k)$ and thus $r(h(k))<r(g_\beta(k))$, so that we get
\begin{equation*}
p_\gamma(h(k))=\omega\cdot q_\gamma(h(k))+r(h(k))<\omega\cdot q_\gamma(g_\beta(k))+r(g_\beta(k))=p_\gamma(g_\beta(k)).
\end{equation*}
One readily derives a contradiction with (v.$\beta$), by considering the sequence
\begin{equation*}
g:(M'_\beta\cap[k])\cup(N\backslash[k])\to X\quad\text{with}\quad g(i):=\begin{cases}
g_\beta(i) & \text{if }i\in M'_\beta\cap[k],\\
h(i) & \text{if }i\in N\backslash[k].
\end{cases}
\end{equation*}
For $\alpha=\beta+1=\gamma+2$ and $k\in N=M'_\alpha\backslash[k_\beta]$, we have seen $q(g_\alpha(k))=q(h(k))\geq\gamma$. We now consider $\alpha=\beta+1$ for a limit~$\beta$. For all $\gamma<\beta$, the given argument yields
\begin{equation*}
q(g_{\gamma+2}(k))\geq\gamma\quad\text{for any}\quad k\in M_{\gamma+2}\backslash[\max(k_{\gamma+1},k_{\gamma+2})].
\end{equation*}
With the aim of setting $N_\beta:=\{i(j)\,|\,j\in\mathbb N\}$, we define $i(0)<i(1)<\ldots\subseteq M_\beta$ as follows: Assuming that $i(j')$ is defined for~$j'<j$, put
\begin{equation*}
\gamma(j):=\sup\{q(t)\,|\, t\in X\text{ with }q(t)<\beta\text{ and }t\ll g_\beta(i(j'))\text{ for some }j'<j\}.
\end{equation*}
Crucially, the new condition~(v) in Definition~\ref{def:gap-order} ensures~$\gamma(j)<\beta$. We now pick a next index $i(j)\in M_\beta$ with $i(j)\geq\max(k_{\gamma(j)+1},k_{\gamma(j)+2})$. Let us note that this choice of~$N_\beta$ is loosely inspired by an argument of Gordeev~\cite{gordeev-gap}. Consider~(\ref{eq_q-beta-to-q}) with $i(j')$ and $i(j)$ at the place of $j$ and $k$, respectively. In the non-trivial case where we have $q_\beta(h(i(j')))=0$, we can invoke $h(i(j'))\ll g_\beta(i(j'))$ to get
\begin{equation*}
q(h(i(j')))\leq\gamma(j)\leq q(g_{\gamma(j)+2}(i(j)))\leq q(g_\beta(i(j)))\leq q(h(i(j))).
\end{equation*}
This completes the proof of~(\ref{eq_q-beta-to-q}) and hence the last case in the recursive construction of sequences $g_\alpha:M_\alpha\to X$ and elements $k_\alpha\in M_\alpha$.

Recall that we aim at a contradiction with the assumption that Theorem~\ref{thm:kriz-minimality} fails for the given~$f$. To conclude his proof, K\v{r}\'{i}\v{z}~\cite{kriz-conjecture} argues that the recursive construction cannot be successful up to an ordinal~$\alpha$ that has uncountable cofinality. Indeed, we have seen that the limit case gives rise to a map $\alpha\ni\beta\mapsto k(\beta)\in\mathbb N$ that is non-decreasing and unbounded. It follows that $\mathbb N\ni k\mapsto\min\{\beta<\alpha\,|\,k(\beta)\geq k\}$~is cofinal in~$\alpha$. Our modified construction cannot even be successful up to~$\alpha=\nu+2$. Otherwise, we would get $q(g_{\nu+2}(k))\geq\nu$ for sufficiently large~$k\in M_{\nu+2}$, as we have seen above. This would yield a value beyond the range of~$q:X\to\nu$.
\end{proof}

To conclude this section, we discuss the formalization of the previous proof. Let us recall that Definition~\ref{def:gap-order} is more restrictive than the notion of `quasiordering with gap-condition' in~\cite{kriz-conjecture}, as we have added the new condition~(v). Thus our version of `K\v{r}\'{i}\v{z}'s minimality principle' is somewhat weaker than the original result in~\cite{kriz-conjecture}. It is still strong enough for a large class of applications, as Theorem~\ref{thm:Kriz-to-KF} demonstrates. We do not know whether the following remains valid for K\v{r}\'{i}\v{z}'s original principle.

\begin{theorem}\label{thm:Pi11-rec-to-Kriz}
K\v{r}\'{i}\v{z}'s minimality principle for~$\nu$ (as formulated in Theorem~\ref{thm:kriz-minimality}) follows from $\Pi^1_1$-recursion along~$\nu$, where~$\nu$ can be any ordinal.
\end{theorem}
\begin{proof}
We build on the notation and arguments from the previous proof. Our first aim is to show that the sequence~$h$ in~(\ref{eq:Lem-2-2}) is arithmetically definable from~$T(\mathcal S)$ and~$\beta$. In particular, we will confirm that~(\ref{eq:Lem-2-2}) is correct. Let us say that a tree~$T$ is perfect if we have $T\subseteq T([T])$ and hence $T=T([T])$ (in particular~$[T]\neq\emptyset$). Note that this is equivalent to the arithmetical condition that~$T$ has no leaves. Given any perfect tree~$T$, we can apply~(\ref{eq:Lem-2-2}) with $\mathcal S:=[T]$, which yields $T=T(\mathcal S)$. Conversely, $T(\mathcal S)$ is perfect for any set~$\mathcal S\neq\emptyset$ of infinite $g:M\to X$, as we have $\mathcal S\subseteq[T(\mathcal S)]$. Under the assumption of~(\ref{eq:Lem-2-2}), we can recover $\mathcal S=[T]$ from $T:=T(\mathcal S)$. So~(\ref{eq:Lem-2-2}) is equivalent to the statement that any perfect tree~$T$ has a branch $g_{T,\beta}\in[T]$ such that $g\ang_\beta g_{T,\beta}$ fails for all~$g\in[T]$, where $\beta$ can be any ordinal. To construct~$g_{T,\beta}$ explicitly, we consider $\sigma\in T$ and write it as $\sigma:\{a_i\,|\,i<k\}\to X$ with $a_0<\ldots<a_{k-1}$~in~$\mathbb N$. We say that $\sigma$ is $\beta$-minimal in~$T$ if the following holds for each index~$j<k$, every number $a'_j\geq\sup\{a_i+1\,|\,i<j\}$ and any sequence $\tau:\{a_i\,|\,i<j\}\cup\{a'_j\}\to X$ in~$T$ such that we have $\sigma(a_i)=\tau(a_i)$ for all $i<j$:
\begin{enumerate}[label=(\roman*)]
\item we have $p_\beta(\sigma(a_j))\leq p_\beta(\tau(a'_j))$,
\item if equality holds in~(i), then we have $a_j\leq a'_j$,
\item if equality holds in~(i) and~(ii), then the numerical code of $\sigma(a_j)$ is smaller than or equal to the one of~$\tau(a'_j)$.
\end{enumerate}
Note that these conditions are arithmetical. Given that~$T$ is perfect, a straightforward induction shows that it contains a unique $\beta$-minimal $\sigma_k:\{a_{k,i}\,|\,i<k\}\to X$ of each length~$k\in\mathbb N$, where we still assume $a_{k,0}<\ldots<a_{k,k-1}$. When we have~$j<k$, the restriction~$\sigma_k[a_{k,j}]$ is $\beta$-minimal and thus equal to~$\sigma_j$. So by uniqueness, the numbers $a_i:=a_{k,i}$ and the values $\sigma_k(a_i)$ do not depend on~$k$. We now consider
\begin{equation*}
g_{T,\beta}:\{a_i\,|\,i\in\mathbb N\}=:M_{T,\beta}\to X\quad\text{with}\quad g_{T,\beta}(a_i):=\sigma_{i+1}(a_i).
\end{equation*}
The construction ensures that we have $g_{T,\beta}\in[T]$ and that $g\ang_\beta g_{T,\beta}$ fails for any branch~$g\in[T]$. Due to the considerations above, we obtain~(\ref{eq:Lem-2-2}) with $h=g_{T(\mathcal S),\beta}$. Crucially, our presentation demonstrates that $g_{T,\beta}:M_{T,\beta}\to X$ is arithmetically definable from~$T$ and~$\beta$. For example, we have $a\in M_{T,\beta}$ precisely when~$a$ lies in the domain of some or every $\beta$-minimal sequence of suitable finite length.

In the notation from the previous proof, let $\mathcal S_\beta$ be the set of all bad $h:N\to X$ with $q(h(i))\geq q(g_\beta(i))$ and $h(i)\ll g_\beta(i)$ for all~$i\in N\in[N_\beta]^\omega$. By considering the recursive construction in this proof, we see that the sequence from condition~(v.$\alpha$) can be given as
\begin{equation*}
g'_{\beta+1}=g_{T,\beta}\quad\text{for}\quad T=T(\mathcal S_\beta).
\end{equation*}
Let us agree that $\overline T(\mathcal S)$ denotes the complement of $T(\mathcal S)$, say, within the set of all finite sequences~$\sigma:a\to X$ with~$a\subseteq\mathbb N$. We argue that $\Pi^1_1$-recursion on~$\alpha$ allows for a simultaneous construction of the sequences $g_\alpha:M_\alpha\to X$, the numbers~$k_\alpha\in M_\alpha$ and the auxiliary sets~$\overline T(\mathcal S_\alpha)$. Indeed, we have just seen that $g'_{\beta+1}$ is arithmetical in $T(\mathcal S_\beta)$ and hence in~$\overline T(\mathcal S_\beta)$. One can conclude that $g_\alpha$ and~$k_\alpha$ are arithmetically definable from the previous stages of the recursion. Also, the relation $h\in\mathcal S_\alpha$ is arith\-metical with parameters $g_\alpha$ and~$k_\beta$ for~$\beta\leq\alpha$ (which determine~$N_\alpha$). It follows that~$T(\mathcal S_\alpha)$ is $\Sigma^1_1$-definable from these parameters, due to the defining equivalence for~$T(\mathcal S)$ in the paragraph before statement~(\ref{eq:Lem-2-2}). Hence the complement~$\overline T(\mathcal S_\alpha)$ is \mbox{$\Pi^1_1$-}definable from the previous~stages. Given that $\Pi^1_1$-recursion is permitted along~$\nu$ and hence along~$\nu+3$ (in the non-trivial case of~$\nu>0$), we can thus carry out the recursive construction from the previous proof. Once this is achieved, it remains to accommodate the inductive verification of conditions~(i.$\alpha)$ to~(v.$\alpha$). Since the latter are $\Pi^1_1$, the induction principle is justified by $\Pi^1_1$-recursion (along~$1\leq\nu$).
\end{proof}

\section{From Kruskal-Friedman theorem to {$\Pi^1_1$}-recursion}

In this section, we prove the remaining implication of Theorem~\ref{thm:main}, which asserts that the uniform Kruskal-Friedman theorem entails $\Pi^1_1$-recursion. A sketch of the proof can be found in the introduction.

The following result allows us to work over a stronger base theory. It will be superseded by the aforementioned implication from Theorem~\ref{thm:main}.

\begin{lemma}\label{lem:KF-to-ACA}
If the uniform Kruskal-Friedman theorem holds for some~$\nu>0$, then we get arithmetical comprehension, over the base theory~$\rca_0+\mathsf{CAC}$.
\end{lemma}
\begin{proof}
Using Corollary~\ref{cor:KF-monotone}, we can reduce the claim to the case of~$\nu=1$, which is covered by~\cite[Lemma~4.6]{frw-kruskal}. We recall the proof of the cited lemma, so that the reader can compare it with the next proof below. Let us write $\seq(Z)$ for the set of finite sequences~$\sigma=\langle z_1,\ldots,z_n\rangle$ with entries~$z_i\in Z$. For another element~$z_0\in Z$ and $\sigma\in\seq(Z)$ as given, we put ${z_0}^\frown\sigma:=\langle z_0,\ldots,z_n\rangle$. If~$Z$ is a partial order, we declare that $\langle\rangle\leq_{\seq(Z)}\sigma$ holds for all~$\sigma\in\seq(Z)$ and that we have
\begin{equation*}
{z_0}^\frown\sigma_0\leq_{\seq(Z)}{z_1}^\frown\sigma_1\quad\Leftrightarrow\quad(z_0\leq_Z z_1\text{ and }\sigma_0\leq_{\seq(Z)}\sigma_1)\text{ or }{z_0}^\frown\sigma_0\leq_{\seq(Z)}\sigma_1.
\end{equation*}
This is a recursive characterization of the order from Higman's lemma, which asserts that $\seq(Z)$ is a well partial order whenever the same holds for~$Z$. The latter is equivalent to arithmetical comprehension, as shown by Simpson~\cite{simpson-higman}. To complete the proof, we show that Higman's lemma follows from the uniform Kruskal-Friedman theorem for~$\nu=1$. Assume~$Z$ is a well partial order. For each partial order~$X$, put
\begin{equation*}
W[Z](X):=1+Z\times X=\{0\}\cup\{(z,x)\,|\,z\in Z\text{ and }x\in X\}.
\end{equation*}
Let us agree that the only inequalities in~$W[Z](X)$ are $0\leq 0$ and $(z,x)\leq(z',x')$ for $z\leq_Z z'$ and $x\leq_X x'$. For a quasi embedding $f:X\to Y$, we define functions
\begin{equation*}
W[Z](f):W[Z](X)\to W[Z](Y)\quad\text{and}\quad\supp_X:W[Z](X)\to[X]^{<\omega}
\end{equation*}
by setting $W[Z](f)(0):=0$ and $W[Z](f)(z,x):=(z,f(x))$ as well as $\supp_X(0):=\emptyset$ and $\supp_X(z,x):=\{x\}$. It is straightforward to check that this turns $W[Z]$ into a normal $\po$-dilator. To learn that we have a $\wpo$-dilator, we invoke the chain antichain principle, which ensures that the well partial orders are closed under taking products (see~\cite{cholak-RM-wpo}). Now consider the function
\begin{equation*}
1\times W[Z](\seq(Z))\cong W[Z](\seq(Z))\xrightarrow{\quad\kappa\quad}\seq(Z)
\end{equation*}
with $\kappa(0):=\langle\rangle$ and $\kappa(z,\sigma):=z^\frown\sigma$. Our recursive characterization of Higman's order coincides with the equivalence from Definition~\ref{def:nu-FP}. Hence $\seq(Z)$ and~$\kappa$ form a $1$-Kruskal fixed point of~$W[Z]$. This fixed point is initial, since the usual length function $h:\seq(Z)\to\mathbb N$ validates condition~(ii) of Theorem~\ref{thm:initial-criterion}. Now the uniform Kruskal-Friedman theorem entails that $\seq(Z)$ is a well partial order.
\end{proof}

We do not know if $\mathsf{CAC}$ follows from the uniform Kruskal-Friedman theorem for fixed~$\nu$, even though the theorem has enormous strength when $\mathsf{CAC}$ is present. At the same time, the next result shows that $\mathsf{CAC}$ is not needed when we admit arbitrary~$\nu$, as the latter can then replace the order~$Z$ from the previous proof. This accounts for the weaker base theory in Corollary~\ref{cor:FP-TR}.

\begin{lemma}\label{lem:bootstrap-KF-univ}
If the uniform Kruskal-Friedman theorem holds for all~$\nu$, then we get arithmetical comprehension (and in particular~$\mathsf{CAC}$), over the base theory~$\rca_0$.
\end{lemma}
\begin{proof}
Let $\omega(\nu)$ be the set of sequences $\langle\alpha_0,\ldots,\alpha_{n-1}\rangle$ with $\nu>\alpha_0\geq\ldots\geq\alpha_{n-1}$ (think of Cantor normal forms). The lexicographic order on~$\omega(\nu)$ can then be characterized by $\langle\rangle\leq_{\omega(\nu)}\sigma$ and
\begin{equation*}
\alpha^\frown\sigma\leq_{\omega(\nu)}\beta^\frown\tau\quad\Leftrightarrow\quad \alpha<\beta\text{ or }(\alpha=\beta\text{ and }\sigma\leq_{\omega(\nu)}\tau).
\end{equation*}
Arithmetical comprehension is equivalent to the statement that this order is well founded for every ordinal~$\nu$, by a result of Girard~\cite[Section~II.5]{girard87} and Hirst~\cite{hirst94}. To apply the uniform Kruskal-Friedman theorem, consider the transformation
\begin{equation*}
X\mapsto W(X):=1+X=\{0\}\cup\{1+x\,|\,x\in X\}
\end{equation*}
of partial orders, where we declare that $0\leq_{W(X)}0$ and $1+x\leq_{W(X)}1+y$ for $x\leq_X y$ are the only inequalities in~$W(X)$. For a quasi embedding $f:X\to Y$, we define $W(f):W(X)\to W(Y)$ and $\supp_X:W(X)\to[X]^{<\omega}$ by setting
\begin{align*}
f(0)&:=0, & \supp_X(0)&:=\emptyset,\\
f(1+x)&:=1+f(x), & \supp_X(1+x)&:=\{x\}.
\end{align*}
It is straightforward to check that this turns $W$ into a normal $\wpo$-dilator. By~the uniform Kruskal-Friedman theorem for~$\nu$, we get a well partial order~$X$ and a bijection $\kappa:\nu\times W(X)\to X$ that validates the equivalence from Definition~\ref{def:nu-FP}. Assuming $\nu>0$, we define $f:\omega(\nu)\to X$ by the recursive clauses
\begin{equation*}
f(\langle\rangle):=\kappa(0,0)\quad\text{and}\quad f(\alpha^\frown\sigma):=\kappa(\alpha,1+f(\sigma)).
\end{equation*}
One can show that $f(\sigma)\leq_X f(\tau)$ entails $\sigma\leq_{\omega(\nu)}\tau$, by a straightforward induction on the combined lenghts of~$\sigma$ and $\tau$ (for $\tau=\beta^\frown\rho$ prove $\rho\leq_{\omega(\nu)}\tau$ as preparation). Given that $X$ is a well partial order, it follows that $\omega(\nu)$ is well founded, as needed to get arithmetical comprehension via the cited result of Girard and Hirst.
\end{proof}

As indicated in the introduction, we now connect the linear and the partial~case. A quasi embedding $f:X\to Y$ of a linear order~$X$ into a partial order~$Y$ can be seen as a linearization of its range. If we know that~$Y$ is a well partial order, then we can conclude that~$X$ is a well order. To exploit this crucial fact, we shall extend the notion from single orders to order transformations, i.\,e., to dilators. By a predilator~$D$ on linear orders we mean a predilator in the sense of~\cite[Definition~1.1]{FR_Pi11-recursion} (which is the obvious linear counterpart of Definition~\ref{def:po-dilator} above). If~$D(X)$ is well founded for any well order~$X$, we say that $D$ is a dilator on linear orders.

\begin{definition}
Consider a $\po$-dilator~$W$ and a predilator~$D$ on linear orders. A quasi embedding $\eta:D\Rightarrow W$ consists of a quasi embedding $\eta_X:D(X)\to W(X)$ for each linear order~$X$, such that the naturality condition $W(f)\circ\eta_X=\eta_Y\circ D(f)$ is satisfied for any embedding~$f:X\to Y$ of linear orders.
\end{definition}

Recall that we write $\mathcal TW[\nu]$ to denote the initial $\nu$-Kruskal fixed point of a normal $\po$-dilator~$W$. For each predilator~$D$ on linear orders, we have a unique linear order~$\psi_\nu(D)$ that is a $\nu$-fixed point in the sense of~\cite[Definition~1.4]{FR_Pi11-recursion}, due to Corollary~2.2 and Theorem~2.9 from the same reference. The cited definition of $\nu$-fixed points in the linear case will be recalled at the beginning of the following proof. Note that a related result for the simpler case~$\nu=1$ has been shown in~\cite{frw-kruskal}.

\begin{theorem}\label{thm:fp-lin-part}
Given a quasi embedding of a predilator~$D$ on linear orders into a normal~$\po$-dilator~$W$, we obtain a quasi embedding of~$\psi_\nu(D)$ into~$\mathcal T W[\nu]$.
\end{theorem}
\begin{proof}
According to~\cite[Definition~1.4]{FR_Pi11-recursion}, the $\nu$-fixed point $\psi_\nu(D)$ comes with a map
\begin{equation*}
\pi:\psi_\nu(D)\to\nu\times D(\psi_\nu(D))
\end{equation*}
that satisfies certain conditions. In order to state the latter, we consider the support functions $\supp^D_X:D(X)\to[X]^{<\omega}$ that are associated with the given predilator~$D$ (see \cite[Definition~1.1]{FR_Pi11-recursion}). Let the binary relation $\tl$ on $Y:=\psi_\nu(D)$ be given by
\begin{equation*}
s\tl t\quad :\Leftrightarrow\quad s\in\supp^D_Y(\tau)\text{ for }\pi(t)=(\alpha,\tau).
\end{equation*}
By the first condition from~\cite[Definition~1.4]{FR_Pi11-recursion}, this relation $\tl$ must be well founded. For any $\gamma<\nu$, we can thus use recursion along $\tl$ to define $G^D_\gamma:Y\to[D(Y)]^{<\omega}$ and simultaneously $G_\gamma:D(Y)\to[D(Y)]^{<\omega}$ with
\begin{align*}
G^D_\gamma(t)&:=\begin{cases}
\{\tau\}\cup G_\gamma(\tau) & \text{if $\pi(t)=(\alpha,\tau)$ with $\alpha\geq\gamma$},\\
\emptyset & \text{if $\pi(t)=(\alpha,\tau)$ with $\alpha<\gamma$},
\end{cases}\\
G_\gamma(\tau)&:=\bigcup\{G^D_\gamma(s)\,|\,s\in\supp^D_Y(\tau)\}.
\end{align*}
Now the final and crucial condition from~\cite[Definition~1.4]{FR_Pi11-recursion} asserts that we have
\begin{equation*}
\rng(\pi)=\{(\alpha,\tau)\in\nu\times D(Y)\,|\,G_\alpha(\tau)\lef_{D(Y)}\tau\}.
\end{equation*}
It may be instructive to consider the similarity with Definition~\ref{def:fp-gap-order}. The rest of the present proof will shed further light on the significance of~$G_\alpha$. Let us also point out that the well foundedness of~$\tl$ relates to criterion~(ii) from Theorem~\ref{thm:initial-criterion}. Indeed, all $\nu$-fixed points of dilators on linear orders are initial by~\cite[Proposition~2.1]{FR_Pi11-recursion}.

Let us write $\eta:D\Rightarrow W$ for the given quasi embedding. As indicated in the introduction, we shall define a quasi embedding $f$ such that the diagram
\begin{equation*}
\begin{tikzcd}[column sep=large]
\psi_\nu(D)\ar[rr,"f"]\ar[d,swap,"\pi"] & & \mathcal T W[\nu]\\
\nu\times D(\psi_\nu(D))\ar[r,"\nu\times\eta_Y"] & \nu\times W(\psi_\nu(D))\ar[r,"\nu\times W(f)"] & \nu\times W(\mathcal T W[\nu])\ar[u,swap,"\kappa"]
\end{tikzcd}
\end{equation*}
commutes, for $Y=\psi_\nu(D)$ and $(\nu\times g)(\alpha,\tau)=(\alpha,g(\tau))$. The point is that this diagram can be rewritten as a recursive equation. To see this, consider $t\in\psi_\nu(D)$ and write $\pi(t)=(\alpha,\tau)$. Due to the support condition from \cite[Definition~1.1]{FR_Pi11-recursion}, we have a unique normal form
\begin{equation*}
\tau\nf D(\iota_a)(\tau')\quad:\Leftrightarrow\quad \tau=D(\iota_a)(\tau')\text{ with }a=\supp^D_Y(\tau)\text{ and }\tau'\in D(a),
\end{equation*}
where $\iota_a:a\hookrightarrow Y$ denotes the inclusion. In view of $\eta_Y\circ D(\iota_a)=W(\iota_a)\circ\eta_a$, our diagram commutes precisely when we have
\begin{equation*}
f(t)=\kappa(\alpha,W(f\circ\iota_a)\circ\eta_a(\tau'))\qquad\text{for }\pi(t)=(\alpha,\tau)\text{ with }\tau\nf D(\iota_a)(\tau').
\end{equation*}
As $s\in a$ means $s\tl t$, this equation can be read as a recursive clause, which is satisfied by at most one~$f$. Concerning existence, we note that our clause can only be applied when $f\circ\iota_a$ is a quasi embedding, so that $W(f\circ\iota_a)$ is defined. To complete the recursive definition regardless, we agree to assign some default value~$f(t)\in\mathcal T W[\nu]$ in the hypothetical case that $f\circ\iota_a$ is no quasi embedding.

We want to show that this hypothetical case does indeed not occur and, in doing so, that $f$ is a quasi embedding. For this purpose, we first use recursion over~$\tl$ to define a length function $l:\psi_\nu(D)\to\mathbb N$ with
\begin{equation*}
l(t):=1+\textstyle\sum_{s\in a}2\cdot l(s)\qquad\text{for }\pi(t)=(\alpha,\tau)\text{ and }a=\supp^D_Y(\tau).
\end{equation*}
We shall now use induction on $l(t_0)+l(t_1)$ to prove
\begin{equation*}
f(t_0)\leq_{\mathcal T W[\nu]} f(t_1)\quad\Rightarrow\quad t_0\leq_Y t_1.
\end{equation*}
Let us write $\pi(t_i)=(\alpha_i,\tau_i)$ and $\tau_i\nf D(\iota_{a(i)})(\tau'_i)$. By the induction hypothesis, $f$~is a quasi embedding on~$a(0)\cup a(1)$. In particular, we are not in the hypothetical case that was mentioned above, so that $f(t_i)$ is defined by the intended recursive clause. This means that $f(t_0)\leq f(t_1)$ amounts to
\begin{equation*}
\kappa(\alpha_0,W(f\circ\iota_{a(0)})\circ\eta_{a(0)}(\tau'_0))\leq_{\mathcal T W[\nu]}\kappa(\alpha_1,W(f\circ\iota_{a(1)})\circ\eta_{a(1)}(\tau'_1)).
\end{equation*}
By the equivalence from Definition~\ref{def:nu-FP}, we must have~$\alpha_0\leq\alpha_1$. If the inequality is strict, we get $\pi(t_0)<\pi(t_1)$ in the linear order $\nu\times D(Y)$ and hence $t_0<_Y t_1$ in~$Y$, as~$\pi$ is an embedding by~\cite[Definition~1.4]{FR_Pi11-recursion}. In the following we assume~$\alpha_0=\alpha_1$. Let us begin with the case where $f(t_0)\leq f(t_1)$ holds by the first disjunct in the equivalence from Definition~\ref{def:nu-FP}, so that we have
\begin{equation*}
W(f\circ\iota_{a(0)})\circ\eta_{a(0)}(\tau'_0)\leq_{W(\mathcal T W[\nu])}W(f\circ\iota_{a(1)})\circ\eta_{a(1)}(\tau'_1).
\end{equation*}
We factor the inclusions $\iota_{a(i)}$ as in the commutative diagram
\begin{equation*}
\begin{tikzcd}[every label/.append style = {font = \small},column sep=large]
a(i)\ar[r,hook,"\iota_{a(i)}"]\ar[d,hook,swap,"{\iota_{a(i)}'}"] & Y\\
a(0)\cup a(1).\ar[ur,hook,swap,"\iota"] &
\end{tikzcd}
\end{equation*}
As noted above, the induction hypothesis ensures that~$f\circ\iota$ is a quasi embedding. We can thus form $W(f\circ\iota)$ and compute
\begin{equation*}
W(f\circ\iota_{a(i)})\circ\eta_{a(i)}(\tau'_i)=W(f\circ\iota)\circ\eta_{a(0)\cup a(1)}\circ D(\iota'_{a(i)})(\tau'_i).
\end{equation*}
Since $W(f\circ\iota)\circ\eta_{a(0)\cup a(1)}$ is a quasi embedding while $D(\iota)$ is an embedding, we get
\begin{equation*}
\tau_0=D(\iota)\circ D(\iota'_{a(0)})(\tau'_0)\leq_{D(Y)}D(\iota)\circ D(\iota'_{a(1)})(\tau'_1)=\tau_1.
\end{equation*}
Once again this yields $\pi(t_0)\leq\pi(t_1)$ and hence~$t_0\leq t_1$. We now consider the case where the second disjunct from Definition~\ref{def:nu-FP} applies, which means that we have
\begin{equation*}
f(t_0)\leqf_{\mathcal T W[\nu]}{\supp^W_{\mathcal T W[\nu]}}\circ W(f\circ\iota_{a(1)})\circ\eta_{a(1)}(\tau'_1).
\end{equation*}
Here we write $\supp^W_X:W(X)\to[X]^{<\omega}$ for the support functions that come with the $\po$-dilator~$W$. As explained in the paragraph after Definition~\ref{def:po-dilator}, the supports are uniquely determined by~$W$ as a functor. It follows that any quasi embedding such as~$\eta:D\Rightarrow W$ respects supports, as shown in~\cite[Lemma~4.2]{frw-kruskal}. Together with naturality, we obtain
\begin{multline*}
{\supp^W_{\mathcal T W[\nu]}}\circ W(f\circ\iota_{a(1)})\circ\eta_{a(1)}(\tau'_1)=[f\circ\iota_{a(1)}]^{<\omega}\circ{\supp^W_{a(1)}}\circ\eta_{a(1)}(\tau'_1)=\\
[f\circ\iota_{a(1)}]^{<\omega}\circ\supp^D_{a(1)}(\tau'_1)=[f]^{<\omega}\circ{\supp^D_Y}(\tau_1).
\end{multline*}
We thus get $f(t_0)\leq f(t)$ for some $t\in\supp^D_Y(\tau_1)=a(1)$. Let us put
\begin{equation*}
c:=\{s\in Y\,|\,l(s)\leq l(s')\text{ for some }s'\in a(0)\cup a(1)\}.
\end{equation*}
Alternatively, one could take~$c$ to be the transitive closure of~$a(0)\cup a(1)$ under~$\tl$ (which can be seen as a subterm relation). Write $\iota':c\hookrightarrow Y$ for the inclusion and observe that $f\circ\iota'$ is a quasi embedding by induction hypothesis (due to the factor~$2$ in the definition of~$l$). We will show that any $s\in c$ validates the following:
\begin{equation}\label{eq:linearize-gap}\tag{$\mathsection$}
\parbox{.9\textwidth}{Whenever we have $\kappa(\gamma,\rho)\leq_{\mathcal T W[\nu]} f(s)$, we get
\begin{equation*}
\rho\leq_{W(\mathcal T W[\nu])} W(f\circ\iota')\circ\eta_c(\sigma)
\end{equation*}
for some $\sigma\in D(c)$ with $D(\iota')(\sigma)\in G^D_\gamma(s)$.}
\end{equation}
Before we prove this statement, we show how the desired inequality $t_0\leq t_1$ follows. To this end, we apply~(\ref{eq:linearize-gap}) with $\kappa(\gamma,\rho):=f(t_0)$ and $s:=t\in a(1)\subseteq c$ from above. Let us extend our earlier diagram of inclusions into
\begin{equation*}
\begin{tikzcd}[every label/.append style = {font = \small},column sep=large]
a(i)\ar[r,hook,"\iota_{a(i)}"]\ar[d,hook,swap,"{\iota_{a(i)}'}"] & Y\\
a(0)\cup a(1)\ar[ur,hook,"\iota"]\ar[r,hook,"{\iota''}"] & c.\ar[u,swap,hook,"{\iota'}"]
\end{tikzcd}
\end{equation*}
Due to the definition of $f(t_0)$, we get $\gamma=\alpha_0=\alpha_1=:\alpha$ as well as
\begin{equation*}
\rho=W(f\circ\iota_{a(0)})\circ\eta_{a(0)}(\tau'_0)=W(f\circ\iota')\circ\eta_c\circ D(\iota''\circ\iota'_{a(0)})(\tau'_0).
\end{equation*}
For $\sigma$ as provided by~(\ref{eq:linearize-gap}), we can derive
\begin{equation*}
\tau_0=D(\iota_{a(0)})(\tau'_0)=D(\iota')\circ D(\iota''\circ\iota'_{a(0)})(\tau_0)\leq_{D(Y)} D(\iota')(\sigma)\in G^D_\alpha(t).
\end{equation*}
In view of the characterization of $\rng(\pi)\ni\pi(t_1)=(\alpha,\tau_1)$ from~\cite[Definition~1.4]{FR_Pi11-recursion}, which was recalled above, we also have
\begin{equation*}
G^D_\alpha(t)\subseteq G_\alpha(\tau_1)\lef_{D(Y)}\tau_1.
\end{equation*}
We obtain $\tau_0\leq D(\iota')(\sigma)<\tau_1$ and hence $\pi(t_0)<\pi(t_1)$ and $t_0<t_1$.

To complete the proof, we establish~(\ref{eq:linearize-gap}) by induction on~$s\in c$ in the well order~$\tl$. Let us write $\pi(s)=(\beta,\sigma_0)$ with $\sigma_0\nf D(\iota_b)(\sigma'_0)$, so that the premise of~(\ref{eq:linearize-gap}) reads
\begin{equation*}
\kappa(\gamma,\rho)\leq_{\mathcal T W[\nu]} f(s)=\kappa(\beta,W(f\circ\iota_b)\circ\eta_b(\sigma'_0)).
\end{equation*}
Analogous to the above, we factor $\iota_b=\iota'\circ\iota'_b:b\hookrightarrow Y$ with $\iota_b':b\hookrightarrow c$. In view of Definition~\ref{def:nu-FP}, we have $\gamma\leq\beta$ and one of the following two cases applies. First assume that we have
\begin{equation*}
\rho\leq_{W(\mathcal T W[\nu])}W(f\circ\iota_b)\circ\eta_b(\sigma'_0)=W(f\circ\iota')\circ\eta_c\circ D(\iota'_b)(\sigma'_0).
\end{equation*}
Then (\ref{eq:linearize-gap}) is valid with $\sigma:=D(\iota'_b)(\sigma'_0)$, as $\pi(s)=(\beta,\sigma_0)$ and $\beta\geq\gamma$ yield
\begin{equation*}
D(\iota')(\sigma)=D(\iota_b)(\sigma_0')=\sigma_0\in G^D_\gamma(s).
\end{equation*}
In the remaining case from Definition~\ref{def:nu-FP}, we have
\begin{equation*}
\kappa(\gamma,\rho)\leqf_{\mathcal T W[\nu]}{\supp^W_{\mathcal T W[\nu]}}\circ W(f\circ\iota_b)\circ\eta_b(\sigma'_0)=[f]^{<\omega}\circ{\supp^D_Y}(\sigma_0),
\end{equation*}
where equality is shown as above. We get $\kappa(\gamma,\rho)\leq f(s')$ for some $s'\in\supp^D_Y(\sigma_0)$. Since the latter entails $s'\tl s$ and in particular~$s'\in c$, we inductively obtain an element $\sigma\in D(c)$ with $\rho\leq W(f\circ\iota')\circ\eta_c(\sigma)$ and
\begin{equation*}
D(\iota')(\sigma)\in G^D_\gamma(s')\subseteq G_\gamma(\sigma_0)\subseteq G^D_\gamma(s),
\end{equation*}
as needed to complete the proof of~(\ref{eq:linearize-gap}).
\end{proof}

Let us deduce the remaining direction of Theorem~\ref{thm:main} from the introduction.

\begin{theorem}
The uniform Kruskal-Friedman theorem for any infinite ordinal~$\nu$ entails $\Pi^1_1$-recursion along the same~$\nu$, over the base theory $\rca_0+\mathsf{CAC}$.
\end{theorem}
\begin{proof}
According to~\cite[Theorem~1.6]{FR_Pi11-recursion}, the principle of $\Pi^1_1$-comprehension along~$\nu$ reduces to the claim that every dilator~$D$ on linear orders has a well founded $\nu$-fixed point in the sense of~\cite[Definition~1.4]{FR_Pi11-recursion}. We note that the restriction to infinite~$\nu$ is inherited from the cited theorem (but see~\cite[Corollary~4.4]{FR_Pi11-recursion} for the case of externally fixed~$\nu<\omega$). To establish the claim about $\nu$-fixed points, we assume the uniform Kruskal-Friedman theorem for~$\nu$ and consider an arbitrary dilator~$D$ on linear orders. One can find a quasi embedding
\begin{equation*}
\eta:D\Rightarrow W
\end{equation*}
into a normal~$\wpo$-dilator~$W$, as shown in~\cite[Section~5]{frw-kruskal} (where dilators on linear orders are called $\mathsf{WO}$-dilators). The cited reference uses arithmetical \mbox{comprehension}, which we can accommodate by Lemma~\ref{lem:KF-to-ACA} above. We note that substantial work goes into the proof that $W$ preserves well partial orders (see~\cite[Theorem~5.11]{frw-kruskal}). Due to Theorem~\ref{thm:fp-lin-part} above, we can convert $\eta$ into a quasi embedding
\begin{equation*}
f:\psi_\nu(D)\to\mathcal T W[\nu],
\end{equation*}
where the linear order $\psi_\nu(D)$ is the unique $\nu$-fixed point of~$D$ while the partial order $\mathcal T W[\nu]$ is the initial $\nu$-Kruskal fixed point of~$W$ (see Section~2 of~\cite{FR_Pi11-recursion} and of the present paper, respectively). So any infinite sequence $t_0,t_1,\ldots$ in $\psi_\nu(D)$ gives rise to a sequence $f(t_0),f(t_1),\ldots$ in~$\mathcal T W[\nu]$. The latter is a well partial order by the uniform Kruskal-Friedman theorem. Thus there must be $i<j$ with $f(t_i)\leq f(t_j)$ and hence $t_i\leq t_j$, which allows us to conclude that $\psi_\nu(D)$ is a well order. As indicated at the beginning of this proof, we can now invoke~\cite[Theorem~1.6]{FR_Pi11-recursion} to secure $\Pi^1_1$-recursion along~$\nu$.
\end{proof}

We have just shown that (i) implies~(ii) in Theorem~\ref{thm:main}. To conclude this paper, we recall how the converse implication is obtained by combining previous results: From Theorem~\ref{thm:Pi11-rec-to-Kriz} we know that $\Pi^1_1$-comprehension along~$\nu$ implies K\v{r}\'{i}\v{z}'s mini\-mality principle for~$\nu$. The latter entails the uniform Kruskal-Friedman theorem for~$\nu$, due to Theorem~\ref{thm:Kriz-to-KF}. The circle of implications shows that $\Pi^1_1$-recursion is equivalent not only to the uniform Kruskal-Friedman theorem but also to K\v{r}\'{i}\v{z}'s minimality principle (up to the paragraph before Theorem~\ref{thm:Pi11-rec-to-Kriz}), which is an interesting result in its own right.

\bibliographystyle{amsplain}
\bibliography{Uniform-Kruskal-Friedman_Freund}

\end{document}